\documentclass[12pt,a4paper,reqno]{amsart}
\usepackage{amssymb}
\usepackage{amsrefs}
\usepackage[all]{xy}
\newtheorem{theorem}{Theorem}[section]
\newtheorem{proposition}[theorem]{Proposition}
\newtheorem{lemma}[theorem]{Lemma}
\newtheorem{corollary}[theorem]{Corollary}

\theoremstyle{definition}
\newtheorem{definition}[theorem]{Definition}
\newtheorem{notation}[theorem]{Notation}

\newtheorem{example}[theorem]{Example}
\newtheorem{observation}[theorem]{Observation}
\newtheorem{remark}[theorem]{Remark}

\numberwithin{equation}{section}

\newcounter{alphcount}
\newenvironment{alphlist}%
{\begin{list}
{(\alph{alphcount})}
{\usecounter{alphcount} \setlength{\labelwidth}{3em}}}%
{\end{list}}

\newcounter{numcount}
\newenvironment{numlist}%
{\begin{list}
{(\arabic{numcount})}
{\usecounter{numcount} \setlength{\labelwidth}{3em}}}%
{\end{list}}

\newenvironment{lnumlist}%
{\begin{list}
{(\arabic{numcount})\hfill}
{\usecounter{numcount} \setlength{\labelwidth}{2em}}}%
{\end{list}}

\newcounter{romcount}
\newenvironment{romlist}%
{\begin{list}
{(\roman{romcount})}
{\usecounter{romcount} \setlength{\labelwidth}{3em}}}%
{\end{list}}

\newcommand\caa{\mathcal {A}}
\newcommand\cbb{\mathcal {B}}
\newcommand\ccc{\mathcal {C}}

\newcommand\cff{\mathcal {F}}

\newcommand\cee{\mathcal {E}}
\newcommand\cgg{\mathcal {G}}
\newcommand\chh{\mathcal {H}}

\newcommand\ckk{\mathcal {K}}
\newcommand\cll{\mathcal {L}}

\newcommand\cnn{\mathcal {N}}
\newcommand\crr{\mathcal {R}}
\newcommand\cww{\mathcal {W}}
\newcommand\cxx{\mathcal {X}}

\let\ideal=\mathfrak

\let\df=\textbf
\let\iso=\cong
\let\adjoint=\dashv
\let\tensor=\otimes
\let\union=\cup
\let\intersection=\cap
\let\Union=\bigcup

\let\htop=\sim
\let\htopt=\approx
\let\inclusion=\hookrightarrow

\DeclareMathOperator\image{im}
\DeclareMathOperator\dom{dom}
\DeclareMathOperator\cod{cod}
\DeclareMathOperator\colim{colim}
\DeclareMathOperator\cell{cell}
\DeclareMathOperator\Hom{Hom}

\newcommand\two{\mathbf {2}}
\newcommand\Set{\mathbf {Set}}
\newcommand\Cat{\mathbf {Cat}}
\newcommand\Mod{\mathbf {Mod}}
\newcommand\Ord{\mathbf {Ord}}
\newcommand\PrOrd{\mathbf {PrOrd}}
\newcommand\rsRel{\mathbf {rsRel}}
\newcommand\eqRel{\mathbf {eqRel}}

\newcommand\Mono{\mathrm {Mono}}
\newcommand\StrEpi{\mathrm {StrEpi}}

\newcommand\cyl{\mathrm {C}}
\newcommand\Ob{\mathrm {Ob}}

\newcommand\komma[2]{(#1{\downarrow} #2)}

\newcommand\follows{\;\Longrightarrow \;}

\newcommand\boxrel{\mathrel{\Box}}
\newcommand\lbox[1]{{}^\Box #1}
\newcommand\rbox[1]{#1^\Box}

\newcommand\natarrow{\overset{\cdot}{\rightarrow}}
\newcommand\map[3]{{#1\colon#2\rightarrow #3}}
\newcommand\natmap[3]{{#1\colon#2\natarrow #3}}

\newcommand\pushout[3]{#1\mathop{+}\limits_{#2}^{\phantom{#2}} #3}
\newcommand\pullback[3]{#1\mathop{\times}\limits_{#2}^{\phantom{#2}} #3}
 
\def\edge[#1]{\ar@{-}@<-0.2pt>[#1] \ar@{-}[#1] \ar@{-}@<+0.2pt>[#1]}

\begin{document}


\title[Model structures for locally presentable categories]%
{Left determined model structures for locally presentable categories}

\author{Marc Olschok}
\address{%
Department of Mathematics\\
Masaryk University\\
Kotl\'a\v{r}sk\'a 2\\
611~37~Brno\\
Czech Republic}
\email{xolschok@math.muni.cz}



\date{\today}


\begin{abstract}
We extend a result of Cisinski on the construction of
cofibrantly generated model structures from (Grothendieck) toposes
to locally presentable categories and from monomorphism to more
general cofibrations.
As in the original case, under additional conditions, the resulting model
structures are ''left determined'' in the sense of Rosick\'y and Tholen.
\end{abstract}

\maketitle

\section{Introduction}

\nobreak\noindent
Given a Quillen model structure on a category, any two of the
three classes of maps involved (cofibrations, fibrations
and weak equivalences) determine the remaining one and hence
the whole model structure. Going one step further, one can ask
for model structures where already one of the classes
determine the other two.

Rosick\'y and Tholen~\cite{ROTH03} introduced the notion of
a left determined model category, where the class $\cww$ of
weak equivalences is determined by the class $\ccc$ of cofibrations
as the smallest class of maps satisfying some closure conditions.
For such a model category, $\cww$ is then the smallest possible
class of weak equivalences for which $\ccc$ and $\cww$ yield a
model structure.

Independently, Cisinski~\cite{CISI02} considered classes of maps
(under the name localizer) that satisfy (almost) the same closure
conditions for the case where the underlying category is a
(Grothendieck) topos and $\ccc$ is the class of monomorphisms.
Moreover, he gave an explicit construction of model structures
for this case, and showed that under suitable conditions the resulting
class of weak equivalences is a smallest localizer (w.r.t.~monomorphisms).
This model structure is then left determined.

Our aim is to extend this construction and the corresponding results
to a more general context, where the class of cofibrations may not be
the monomorphisms and where the underlying category is not necessarily
a topos. The necessary assumptions for such a generalization to work
fall into three sorts:

\begin{itemize}
\item general conditions on the underlying category. We assume, that the
underlying category is locally presentable. Since every Grothendieck
topos is locally presentable, this will include the original examples.

\item conditions on the class of cofibrations in spe.
We assume, that these are already part of a cofibrantly generated
weak factorization system and that every object is cofibrant.

\item conditions on the cylinder used for the construction.
These will be discussed later.
\end{itemize}

The remaining sections of this paper are as follows: Section~2
contains the needed definitions and facts about accessible categories,
weak factorization systems and model structures (mostly without proofs).
In Section~3 we show that, given a cofibrantly generated weak
factorization system together with a cylinder satisfying suitable
assumptions, Cisinski's construction produces a cofibrantly generated model
structure. In Section~4 we compare the weak equivalences produced by
that construction with smallest localizers and identify conditions under
which these coincide. Finally, Section~5 contains some well known examples
in order to illustrate the construction. For the case of module categories
we also describe the used cylinders in terms of pure submodules.

Notation is almost standard; but we write composition in reading order
and denote identity morphisms by the name of their objects.

\section{Accessible categories and model structures}

\nobreak\noindent
We first turn to accessible and locally presentable categories.
The main source for this material is the book
of Ad\'amek and Rosick\'y~\cite{ADRO94}.

\begin{definition} 
Let $\lambda$ be a regular cardinal.
\begin{alphlist}
\item
an object $X$ in a category $\ckk$ is \df{$\lambda$-presentable}
if the functor $\map{\ckk(X,-)}\ckk\Set$ preserves
$\lambda$-directed colimits.

\item
A category $\ckk$ is \df{$\lambda$-accessible} if it satisfies
the following two conditions:
\begin{numlist}
\item
$\ckk$ has $\lambda$-directed colimits.

\item
there is a set $\caa$ of $\lambda$-presentable objects of $\ckk$ such
that every object of $\ckk$ is a $\lambda$-directed colimit of objects
from $\caa$.
\end{numlist}
It is \df{accessible} if it is $\lambda$-accessible for
some regular cardinal $\lambda$.

\item
A category $\ckk$ is \df{$\lambda$-locally presentable} if it is
$\lambda$-accessible and cocomplete. It then follows that it is
also complete, see e.g.~\cite{ADRO94}*{Corollary~1.28}.
It is \df{locally presentable} if it is $\lambda$-presentable for
some regular cardinal $\lambda$.

\item
A functor $\map F\ckk\cll$ is $\lambda$-accessible if both
$\ckk$ and $\cll$ are $\lambda$-accessible and $F$ preserves
$\lambda$-directed colimits.
It is \df{accessible} if it is $\lambda$-accessible for
some regular cardinal $\lambda$.

\item
A full subcategory $\ckk$ of $\cll$ is \df{accessibly embedded}
if it is closed under $\lambda$-directed colimits for some
regular cardinal $\lambda$.
\end{alphlist}
\end{definition}

\begin{notation}
Let $\map F\caa\cbb$ be any functor.
We write $F\caa$ for the \df{full image} of $\caa$ under $F$,
i.e.~the full subcategory of $\cbb$ determined by all
objects $FX$ ($X\in\caa)$.
If $\ckk$ is a full subcategory of $\cbb$, we write
$F^{-1}\ckk$ for its \df{full preimage} under $F$,
i.e.~the full subcategory of $\caa$ determined by all
those objects $X\in\caa$ with $FX\in \ckk$.
\end{notation}

\begin{lemma}
\label{lem:preimage_of_replete_reductive}
Let $\map F\caa\ccc$ be an accessible functor and let $\ckk$ be
a full subcategory of $\ccc$.
\begin{alphlist}

\item
If $\ckk$ is accessible and accessibly embedded in $\ccc$ then
$F^{-1}\ckk$ is also accessible and accessibly embedded in $\caa$.

\item
If $\ckk$ is the full image of an accessible functor and also
isomorph\-ism-closed in $\ccc$ then the same holds for $F^{-1}\ckk$.
\end{alphlist}
\end{lemma}

\begin{proof}
Part (a) is \cite{ADRO94}*{Remark~2.50}.
For part (b), let $\map G\cbb\ccc$ be an accessible functor with
$\ckk=G\cbb$.
\begin{numlist}
\item
The comma category $\komma FG$ is accessible and the projection
$\komma FG \rightarrow \caa$ is accessible by \cite{ADRO94}*{Theorem~2.43}.

\item
$F$ and $G$ induce an accessible functor $\map H{\komma FG}{\ccc^\two}$ via
\newline
$H(A,B,\map u{FA}{GB}) = u$.
Since the full subcategory of $\ccc^\two$ given by isomorphisms
is accessible and accessibly embedded in $\ccc^\two$, the same
holds for its preimage under $H$, by part (a). This preimage
is the full subcategory $\mathrm{Iso}(F,G)$ of $\komma FG$ whose objects
are those $(A,B,\map u{FA}{GB})$ for which $u$ is an isomorphism.

\item
$F^{-1}(G\cbb)$ is the full image of the composite
\newline
$\mathrm{Iso}(F,G) \inclusion \komma FG \rightarrow \caa$.\qedhere
\end{numlist}
\end{proof}

We now turn to model structures. We follow Ad\'amek, Herrlich, Rosick\'y,
Tholen~\cite{AHRT02b} in introducing these via the notion of a weak
factorization system. Other sources include the article of Beke~ \cite{BEKE00}
and the books of Hirschhorn~\cite{HIRS03} and Hovey~\cite{HOVE99}.
Most definitions do not need the underlying category to be complete
and cocomplete as is usually assumed when working with model structures.
For now we tacitly assume that the relevant limits and colimits exist
for the various statements to make sense.

\begin{notation}
For two maps $f$ and $g$ in a category $\ckk$ we write
$f\boxrel g$ if for every solid square
$$
\xymatrix{
{\bullet} \ar[r] \ar[d]_f & {\bullet} \ar[d]^g \\
{\bullet} \ar[r] \ar@{.>}[ru] & {\bullet}
}
$$
the (dotted) diagonal exists. For a class $\chh$ of maps we set
\newline
$\rbox\chh = \{ g\in\ckk \mid \forall h\in\chh\colon h\boxrel g \}$
and
$\lbox\chh = \{ f\in\ckk \mid \forall h\in\chh\colon f\boxrel h \}$.
\end{notation}

\begin{remark}
\begin{lnumlist}
\item \hspace*{-9pt}Any class of the form  $\lbox\chh$ is
stable under pushouts, retracts in $\ckk^\two$ and transfinite
compositions of smooth chains,
where a smooth chain is a colimit preserving functor $\map D\alpha\ckk$
from some ordinal and its transfinite composition is the induced map
from $D_0$ to $\colim_{\beta<\alpha}D_{\beta}$.
The dual results hold for classes of the form $\rbox\chh$.
We write $\cell(\chh)$ for the class of those maps that are
transfinite compositions of pushouts of maps from $\chh$.
Hence the above observation in particular gives
$\cell(\chh) \subseteq \lbox{(\rbox\chh)}$.
\end{lnumlist}
\begin{numlist}
\setcounter{numcount}{1}
\item
Suppose $f=xy$.
If $f\boxrel y$, then by redrawing
$$
\xymatrix{
{\bullet} \ar[r]^x \ar[d]_f & {\bullet} \ar[d]^y \\
{\bullet} \ar@{=}[r] \ar[ru]^d & {\bullet}
}
\xymatrix{
{\quad\quad\quad}\ar@{}[d]|{\mbox{as}}\\ { }
}
\xymatrix{
{\bullet} \ar[d]_f \ar@{=}[r]
& {\bullet} \ar[d]^x \ar@{=}[r] & {\bullet} \ar[d]^f \\
{\bullet} \ar[r]_d & {\bullet} \ar[r]_y & {\bullet}
}
$$
one obtains $f$ as a retract of $x$.
Dually, if $x\boxrel f$ then $f$ is a retract of $y$.

\item
The relation $\Box$ gives a Galois-connection on classes of maps,
i.e.~one always has $\cll\subseteq\lbox\crr \iff \rbox\cll\supseteq\crr$.
\end{numlist}
\end{remark}

\begin{definition}
\label{def:wfs}
A \df{weak factorization system} in a category $\ckk$
is a pair $(\cll,\crr)$ of classes of maps such that the
following two conditions are satisfied:
\begin{numlist}
\item
$\cll=\lbox\crr$ and $\rbox\cll=\crr$.

\item
Every map $f$ has a factorization as $f=\ell r$ with
$\ell\in\cll$ and $r\in\crr$.
\end{numlist}
The weak factorization system $(\cll,\crr)$ is \df{cofibrantly generated}
if $\cll=\lbox{(\rbox I)}$ for some subset $I\subseteq\cll$.
It is \df{functorial} if there is a functor $\map F{\ckk^\two}\ckk$
together with natural maps $\natmap \lambda\dom{F}$ and $\natmap \rho F\cod$
such that $\lambda_f\in\cll$, $\rho_f\in\crr$ and $f=\lambda_f\rho_f$
for all $f\in\ckk^\two$.
\end{definition}

\begin{definition}
\label{def:mods}
A \df{model structure} $(\ccc,\cww,\cff)$ on a category
$\ckk$ consists of three classes $\ccc$ (cofibrations),
$\cff$ (fibrations) and $\cww$ (weak equivalences) such that
the following conditions are satisfied:
\begin{numlist}
\item
$\cww$ is closed under retracts in $\ckk^\two$ and has the
\df{2-3 property}: if in $f=gh$ two of the maps
lie in $\cww$ then so does the third.

\item
Both $(\ccc,\cww\intersection\cff)$
and $(\ccc\intersection\cww,\cff)$
are weak factorization systems.
\end{numlist}
The classes $\ccc\intersection\cww$ and $\cww\intersection\cff$
are called trivial cofibrations and trivial fibrations respectively.
The model structure is \df{cofibrantly generated} or \df{functorial}
if the two weak factorization systems in (2) are.
An object $X$ is called $\df{cofibrant}$ if the map $(0\to X)$
from the initial object is a cofibration and $\df{fibrant}$ if
the map $(X\to 1)$ to the terminal object is a fibration.
For a functorial model structure, one obtains the
\df{cofibrant replacement functor} and the
\df{fibrant replacement functor} by restricting the two functorial
factorizations to $\komma0{\ckk}$ and $\komma{\ckk}1$ respectively.
\end{definition}

\begin{remark}
\label{rem:trivial}
Any weak factorization system $(\cll,\crr)$ in $\ckk$ gives a model
structure with $\ccc=\cll$, $\cff=\crr$ and $\cww=\ckk$ for which
Definition~\ref{def:wfs} and Definition~\ref{def:mods} produce the same
notions of ''cofibrantly generated'' and ''functorial''.
Any notion about model structures in general
(like e.g.~''(co)fibrant objects'' or ''(co)fibrant replacement functor''
from above) can be applied to weak factorization systems by considering
this special model structure.
\end{remark}

\begin{definition}
\label{def:htop}
Let $(\ccc,\cww,\cff)$ be a model structure in a category $\ckk$.
\begin{alphlist}
\item
For an object $X$, a \df{cylinder object} $\cyl X$ for $X$ is
given by a $(\ccc,\cww)$-factorization of the codiagonal
$\map {(X|X)}{X+X}X$ as
$\xymatrix{ X+X \ar[r]|-{\gamma_X} & {\cyl X} \ar[r]|-{\sigma_X}& X}$.
The cylinder object $\cyl X$ is \df{final}
if $\sigma_X\in \rbox\ccc$.
Given two cylinder objects $\cyl X$ and $\cyl' X$, we call
$\cyl' X$ \df{finer than} $\cyl X$ if there is a
$\map {\varphi_X}{\cyl X}{\cyl' X}$ making the following
diagram commutative:
$$
\xymatrix{
& X+X \ar[ld]_{\gamma_X} \ar[rd]^{\gamma'_X} & \\
{\cyl X} \ar[rr]_{\varphi_X} \ar[rd]_{\sigma_X}
& &
{\cyl' X} \ar[ld]^{\sigma'_X} \\
& X &
}
$$

\item
A \df{(functorial) cylinder} $(\cyl,\gamma,\sigma)$ is a functor
$\map\cyl\ckk\ckk$ together with natural maps $\gamma$ and $\sigma$
whose $X$-components make $\cyl X$ into a cylinder object as in (a).
Together with the (natural) coproduct inclusions one then obtains
natural maps with $X$-components as in the diagram below:
\begin{equation}
\label{dia:cyl}
\xymatrix{
X \ar[r]^-{\iota^0_X} \ar[dr]|{\gamma^0_X} \ar@{=}[ddr]
& X+X \ar[d]|{\gamma_X}
& X \ar[l]_-{\iota^1_X} \ar[dl]|{\gamma^1_X} \ar@{=}[ddl]
\\
& {\cyl X} \ar[d]|{\sigma_X} & \\
& X & 
}
\end{equation}
The cylinder is \df{final} if all $\sigma_X$ are in $\rbox\ccc$.
A cylinder $(\cyl',\gamma',\sigma')$ is \df{finer than}
$(\cyl,\gamma,\sigma)$ iff $\cyl X$ is finer than $\cyl' X$
for all $X$.

\item
Given a cylinder $(\cyl,\gamma,\sigma)$, two maps $\map {f,g}XY$ are
\df{homotopic} if the induced map
$\map {(f|g)}{X+X}Y$ from the coproduct factors through
$\map {\gamma_X}{X+X}{\cyl X}$. This will be written as
$f \htop g$ or sometimes as $f \htop g \pmod\cyl$.

\item
The symmetric transitive closure of $\htop$ is written as $\htopt$.
Since $\htop$ is reflexive and compatible with composition,
$\htopt$ is a congruence relation.
The quotient category will be denoted by $\ckk/{\htopt}$.
A map $\map fXY$ is a \df{homotopy equivalence}, if its
image in $\ckk/{\htopt}$ is an isomorphism, or equivalently,
if there exists a $\map gYX$ with $fg\htopt X$ and $gf\htopt Y$.
\end{alphlist}

\noindent
For a weak factorization system $(\cll,\crr)$, cylinder objects,
functorial cylinders and homotopy are defined as those for the
trivial model structure $(\cll,\ckk,\crr)$.
\end{definition}

\begin{observation}
\label{obs:cylinders}
\label{rem:coinj_implies_cof} 
Let $(\cll,\crr)$ be a weak factorization system (similar observations
apply to model structures).
\begin{alphlist}

\item
Suppose that in part (a) of Definition~\ref{def:htop} the object
$X$ is cofibrant. Then the coproduct injections $\iota^0_X$ and $\iota^1_X$
are in $\cll$, being pushouts of the map $(0\to X)$.
Consequently, not only $\gamma_X$, but also $\gamma^0_X$ and
$\gamma^1_X$ are in $\cll$.

\item
The $(\cll,\crr)$-factorizations of codiagonals provide enough final
cylinder objects and every cylinder object $\cyl X$ can be refined
to a final one by a $(\cll,\crr)$-factorization of $\map {\sigma_X}{\cyl X}X$.
Also every final cylinder object is a finest one: if a $\cyl X$ is final and
$\cyl' X$ is any other cylinder object, then $\gamma'_X\boxrel\sigma_X$
will give a diagonal in
$$
\xymatrix{
{X+X} \ar[d]_{\gamma'_X} \ar[r]^{\gamma_X} & {\cyl X} \ar[d]^{\sigma_X} \\
{\cyl' X} \ar[r]_{\sigma'_X} \ar[ru] & X
}
$$
so that $\cyl X$ is finer than $\cyl' X$.

\item
Suppose that $(\cll,\crr)$ is functorial. Then one always has enough
final cylinders and every cylinder $(\cyl,\gamma,\sigma)$ can be refined
to a final cylinder by a functorial factorization of $\sigma$.

\item
If $(\cyl',\gamma',\sigma')$ is finer than $(\cyl,\gamma,\sigma)$
then the implication
$$ f\htop g \pmod{\cyl'} \follows  f\htop g \pmod\cyl $$
holds for any two maps $\map {f,g}XY$.
In particular any two final cylinders determine the same homotopy relation.
\end{alphlist}
\end{observation}

\begin{remark}
When functorial factorizations are not available, one
can still define homotopy as in Part (c) of Definition~\ref{def:htop}
for a nonfunctorial choice of cylinder objects $\cyl X$ without any
naturality condition on the maps $\gamma_X$ or $\sigma_X$.

One can also relax the definition by not fixing a choice for a cylinder
object: two maps $\map {f,g}XY$ are homotopic if $(f|g)$ factors through
some $\map {\gamma_X}{X+X}{\cyl X}$ of a cylinder object.
This is known as ''left homotopy''  in the literature on model categories
(see e.g.~\cite{HIRS03}*{Definition~7.3.2} or
\cite{HOVE99}*{Definition~1.2.4}).
But the resulting homotopy relation is not necessarily compatible
with precomposition.

An alternative approach is to use a fixed choice of final cylinder
objects. The existence of certain diagonals then works as a substitute
for the missing naturality. The homotopy relation with respect to such
a choice will always be symmetric and compatible with composition.
Moreover (by an argument as in Observation~\ref{obs:cylinders})
it does not depend on the choice of cylinder objects.
This approach was introduced by Kurz and Rosick\'y~\cite{KURO05}.

Since we will only meet situations where functorial factorizations are
available, we will not need this added generality.
\end{remark}

We now turn to weak factorization systems in locally presentable
categories. The following theorem should indicate, why these
categories are a convenient setting.

\begin{theorem}
\label{thm:lp_is_wfs} 
\label{thm:funcfact_is_accessible} 
\label{thm:heq_is_reductive} 
Let $\ckk$ be a locally presentable category and $I$ a set of maps
in $\ckk$.
\begin{alphlist}
\item
Every map $f$ can be factored as $f=xy$ with $x\in\cell(I)$
and $y\in\rbox I$. Moreover this factorization can be made
functorial. In particular $(\lbox{(\rbox I)},\rbox I)$ is a
functorial factorization system.

\item
In the situation of (a), the factorization functor $\ckk^\two\to\ckk$
is accessible.

\item
The full subcategory of $\ckk^\two$ given by the homotopy equivalences
with respect to a final cylinder is the full image of an accessible
functor.

\end{alphlist}
\end{theorem}

\begin{proof}
Part (a) is shown e.g.~in \cite{BEKE00}*{Proposition~1.3}.
Part (b) is due to J.H.~Smith; for a published proof see
e.g.~Rosick\'y~\cite{ROSI05}*{Proposition~3.1}. The statements therein are
phrased for model structures but apply to weak factorization
systems via Remark~\ref{rem:trivial}.
Part (c) is \cite{ROSI07}*{Proposition~3.8}. 
\end{proof}

The last ingredient will be a theorem of Smith which describes
conditions under which two classes $\ccc$ and $\cww$ of maps
in a locally presentable category are part of a cofibrantly
generated model structure.

\begin{definition} A functor $\map F\caa\cbb$ satisfies
\begin{alphlist}
\item
the \df{solution set condition at an object} $B$ of $\cbb$
if there is a set of maps $\{\map {f_i}B{FA_i} \mid i\in I\}$
such that every map $\map fB{FA}$ factors as $f=f_i(Fu)$
for some $f_i$ and $\map u{A_i}{A}$.

\item
the \df{solution set condition at a class of objects}, if it satisfies
the solution set condition at every element of that class.

\item
the \df{solution set condition}, if it satisfies the
solution set condition at all objects of $\cbb$.
\end{alphlist}
\noindent
A full subcategory $\ckk$ of $\cbb$ satisfies the conditions above
if its inclusion functor does.
\end{definition}

\begin{lemma}
\label{lem:reductive_implies_solution-set}
Every accessible functor $\map F\ckk\cll$ (and hence its full
image) satisfies the solution set condition.
\end{lemma}

\begin{proof}
\cite{ADRO94}*{Corollary~2.45}
\end{proof}

\begin{theorem}[Smith's Theorem]
\label{thm:smith}
Let $\ckk$ be a locally presentable category, $I$ a set of maps
and $\cww$ a class of maps in $\ckk$. Suppose that the following
conditions are satisfied:
\begin{numlist}

\item
$\cww$ has the 2-3 property and is closed under
retracts in $\ckk^\two$.

\item
$\rbox I \subseteq \cww$

\item
$\lbox{(\rbox I)} \intersection \cww$ is closed under pushouts
and transfinite composition.

\item
$\cww$ satisfies the solution set condition at $I$.
\end{numlist}

\noindent
Then setting
$\ccc:=\lbox{(\rbox I)}$ and $\cff:=\rbox{(\ccc\intersection\cww)}$
gives a cofibrantly generated model structure $(\ccc,\cww,\cff)$ on $\ckk$.
\end{theorem}

\begin{proof}
\cite{BEKE00}*{Theorem~1.7}
\end{proof}

\begin{remark}
Conditions (1)--(3) in the above Theorem are necessary for any
cofibrantly generated model structure $(\ccc,\cww,\cff)$ with $I$
being the set of generating cofibrations.
Rosick\'y \cite{ROSI07}*{Theorem~4.3} has recently shown that
condition (4) is also necessary.
\end{remark}

\section{Cisinski's construction}

\nobreak\noindent
We now present the construction of a cofibrantly generated model
structure from a suitable cofibrantly generated weak factorization system
and cylinder. As in the original case, we need additional conditions on
the cylinder used. Our conditions in Definition~\ref{def:cart} are different
from those of Cisinski~\cite{CISI02}*{D\'efinition~2.3}.
Nevertheless, they are equivalent in the case of $(\Mono,\rbox\Mono)$
in a Grothendieck topos.

Before turning to the actual construction, we first look at
one particular ingredient in a more general setting.

\begin{definition}
\label{def:star}
Let $\caa$ be a category with pushouts. Given a natural map
$\alpha\colon F\natarrow F'\colon \cxx\rightarrow\caa$ and
a map $\map fXY$ let $f\star\alpha$ be the connecting map
in the diagram below:
$$
\xymatrix{
FX \ar[r]^{\alpha_X} \ar[d]_{Ff} & F'X \ar[d] \ar@/^/[ddr]^{F'f} & \\
FY \ar[r] \ar@/_/[drr]_{\alpha_Y} &
{\pushout{FY}{FX}{F'X}} \ar[dr]|-{f\star\alpha} & \\
 & & F'Y
}
$$
Dually, let $\cxx$ be a category with pullbacks. Given a natural map
$\beta\colon G'\natarrow G\colon \caa\rightarrow\cxx$ and
a map $\map gAB$ let $\beta\star g$ be the connecting map
in the diagram below:
$$
\xymatrix{
G'A
\ar@/_/[ddr]_{G'g} \ar[dr]|-{\beta\star g} \ar@/^/[rrd]^{\beta_A} && \\
& {\pullback{G'B}{GB}{GA}} \ar[r] \ar[d] & GA \ar[d]^{Gg}  \\
& G'B \ar[r]_{\beta_B}& GB
}
$$
For a class $I$ of maps,
we write $I\star\alpha$ for $\{f\star\alpha\mid f\in I\}$
and $\beta\star I$ for $\{\beta\star f\mid f\in I\}$.
\end{definition}

For the next Lemma, recall the notion of a conjugate pair of
natural maps between two adjunctions from e.g.~Mac~Lane\cite{MACL98}*{IV-7}:
given two adjunctions $F:\cxx\rightleftarrows\caa:G$ and
$F':\cxx\rightleftarrows\caa:G'$, two natural maps $\natmap\alpha{F}{F'}$
and $\natmap\beta{G'}{G}$ are conjugate if the diagram
$$
\xymatrix{
\caa(F'X,A) \ar[r]^{\iso} \ar[d]_{\caa(\alpha_X,A)}
& \cxx(X,G'A) \ar[d]^{\cxx(X,\beta_A)} \\
\caa(FX,A) \ar[r]^{\iso} & \cxx(X,GA)
}
$$
commutes for all $X\in\cxx$ and $A\in\caa$.

\begin{lemma}
\label{lem:conjugate}
Suppose $\natmap\alpha{F}{F'}$ and $\natmap\beta{G'}{G}$ are two
conjugate natural maps. Then for all $\map fXY$ and $\map gAB$
one has
$$ (f\star\alpha) \boxrel g \quad\iff\quad f \boxrel (\beta\star g) $$
\end{lemma}

\begin{proof}
We will show the direction ''$\Rightarrow$''. The opposite direction
then follows by duality. So assume $(f\star\alpha) \boxrel g$ and
consider any diagram
$$
\xymatrix{
X \ar[r]^u \ar[d]_f & G'A \ar[d]^{\beta\star g} \ar[dr]^{\beta_A}& \\
Y \ar[r]_v \ar[dr]_{vp} & P \ar[d]^p \ar[r]_q & GA \ar[d]^{Gg} \\
& G'B \ar[r]_{\beta_B} & GB
}
$$
where $P$ is the pullback of $\beta_B$ and $Gg$.
We need a diagonal for the left upper square.
Switching via the adjunctions (indicated by $\widehat{(\phantom{-})}$ in both
directions) gives the solid arrows of the diagram
$$
\xymatrix{
FX \ar[r]^{\alpha_X} \ar[d]_{Ff} & F'X \ar[d]_j \ar[dr]^{\widehat{u}}& \\
FY \ar[r]^i \ar[dr]_{\alpha_Y} \ar@/_1pc/[rr]|-(0.7){\widehat{vq}}
 & Q \ar[d]|-{f\star\alpha} \ar@{.>}[r]^r
 & A \ar[d]^g \\
& F'Y \ar[r]_{\widehat{vp}} \ar@{.>}[ru]_d & B
}
$$
where $Q$ is the pushout of $Ff$ and $\alpha_X$.
Now $\map rQA$ is induced by $\widehat{vq}$ and $\widehat{u}$.
Testing against $i$ and $j$ yields the commutativity of the right
lower square (i.e. $rg=(f\star\alpha)\widehat{vp}$), which therefore
has a diagonal $\map d{F'Y}A$.
Switching back via the adjunction gives
$$
\xymatrix{
X \ar[r]^u \ar[d]_f & G'A \ar[d]^{\beta\star g} \ar[dr]^{\beta_A}& \\
Y \ar[r]_v \ar[dr]_{vp} \ar[ru]^{\widehat{d}}& P \ar[d]^p \ar[r]_q
 & GA \ar[d]^{Gg} \\
& G'B \ar[r]_{\beta_B} & GB
}
$$
where the equality $\widehat{d}(\beta\star g)=v$ can
be verified by testing against $p$ and $q$.
Hence $\map{\widehat{d}}Y{G'A}$ is the desired diagonal.
\end{proof}

\begin{corollary}
\label{cor:starstable}
In the situation of the previous Lemma, let $I$ be a class of maps in $\cxx$
and $J$ be a class of maps in $\caa$.
Then
$$
I\star\alpha \subseteq J
\follows
(\lbox{(\rbox{I})})\star\alpha
\subseteq \lbox{(\rbox{J})}
$$
\end{corollary}

\begin{proof}
\begin{align*}
I\star\alpha \subseteq J
& \follows I\star\alpha \subseteq  \lbox{(\rbox{J})} \\
& \follows \rbox{I} \supseteq \beta\star(\rbox{J}) \\
& \follows \rbox{(\lbox{(\rbox{I})})}
    \supseteq \beta\star(\rbox{J}) \\
& \follows (\lbox{(\rbox{I})})\star\alpha
\subseteq \lbox{(\rbox{J})}\qedhere
\end{align*}
\end{proof}

\begin{remark}
Corollary \ref{cor:starstable} applies to any natural map between
left adjoints (assuming that the necessary pushouts and pullbacks exist)
because any such map determines a conjugate map between the respective
right adjoints.
\end{remark}

\begin{definition}
\label{def:lambda}
Let $(\cll,\crr)$ be a cofibrantly generated weak factorization system
in a locally presentable category $\ckk$. For a functorial cylinder
$(\cyl,\gamma,\sigma)$, a generating set $I$ and a subset 
$S\subseteq \lbox{(\rbox I)}$ define $\Lambda(\cyl,S,I)$
via the following construction:
\begin{align}
&\Lambda^0(\cyl,S,I) := S \union (I\star\gamma^0) \union (I\star\gamma^1) \\
&\Lambda^{n+1}(\cyl,S,I) := \Lambda^n(\cyl,S,I)\star\gamma \\
&\Lambda(\cyl,S,I) := \Union_{n\geq0}\Lambda^n(\cyl,S,I)
\end{align}
\end{definition}

\begin{lemma}
\label{lem:anodext}
Suppose a cylinder functor $\cyl$ for $(\cll,\crr)$ is a left adjoint.
Then for any two generating subsets $I,J \subseteq\cll$ one has
$$\lbox{(\rbox{\Lambda(\cyl,S,I)})} = \lbox{(\rbox{\Lambda(\cyl,S,J)})}$$
\end{lemma}

\begin{proof}
We will drop $\cyl$ and $S$ from the notation for $\Lambda$ and show
\newline
$\Lambda^n(I) \subseteq \lbox{(\rbox{\Lambda(J)})}$
for all $n\geq0$.

\begin{numlist}
\item
We have $J\star\gamma^k\subseteq\Lambda(J)$ (for $k=0,1$).
Corollary~\ref{cor:starstable} then gives
$\cll\star\gamma^k\subseteq \lbox{(\rbox{\Lambda(J)})}$.
So in particular $\Lambda^0(I)\subseteq \lbox{(\rbox{\Lambda(J)})}$.

\item
Assume $\Lambda^n(I) \subseteq \lbox{(\rbox{\Lambda(J)})}$.
\newline
Corollary~\ref{cor:starstable} then gives
$
\Lambda^{n+1}(I)
=\Lambda^n(I)\star\gamma \subseteq \lbox{(\rbox{\Lambda(J)})}
$\qedhere
\end{numlist}
\end{proof}

\begin{remark}
In general one cannot expect $\Lambda(\cyl,S,I)\subseteq\cll$ without any
further assumptions. However, if $\cyl$ is a left adjoint,
Lemma \ref{lem:anodext} shows, that this property does not depend on
the choice of the generating subset.
This motivates the following definition.
\end{remark}

\begin{definition}
\label{def:cart}
Let $(\cll,\crr)$ be weak factorization system in a category $\ckk$.
A functorial cylinder $(\cyl,\gamma,\sigma)$ is \df{cartesian} if 
\begin{alphlist}
\item
The cylinder functor $\map{\cyl}{\ckk}{\ckk}$ is a left adjoint
\item $\cll\star\gamma\subseteq\cll$ and
$\cll\star\gamma^k\subseteq\cll$ \ ($k=0,1$)
\end{alphlist}
\end{definition}

\begin{remark}
Condition (a) allows using Lemma~\ref{lem:conjugate}
and Corollary~\ref{cor:starstable}. In particular, if $(\cll,\crr)$
is cofibrantly generated by some subset $I\subseteq\cll$, Condition (b)
already holds whenever $I\star\gamma^0$, $I\star\gamma^1$ and $I\star\gamma$
lie in $\cll$. Also for any $f\in\cll$ we have $\cyl f=f'(f\star\gamma^0)$
where $f'$ is a pushout of $f$, so that $\cyl f$ is again in $\cll$.
\end{remark}

\noindent
We now insert a comparison of Definition~\ref{def:cart}
with \cite{CISI02}*{D\'efinition~2.3}.
Let $\cee$ be a Grothendieck topos. We recall the following properties:
\begin{numlist}

\item
Colimits in $\cee$ are universal: given a colimit cocone
$\map{x_i}{X_i}X$ and a map $\map fYX$, the induced maps
$\map{f^*(x_i)}{f^*(X_i)}Y$ obtained from pulling back the $x_i$
along $f$ again form a colimit cocone.
This is \cite{JOHN77}*{Lemma~1.51}.

\item
$\cee$ is locally presentable. This follows from \cite{ADRO94}*{Theorem~1.46}
together with the fact that the sheaves with respect to a site form a
small orthogonality class (in the sense of \cite{ADRO94}*{Definition~1.35})
inside the respective presheaf topos.

\item
Whenever one has a diagram
\begin{equation}
\label{dia:eff}
\xymatrix{
P \ar[d]_a \ar[r]^b & B \ar[d] \ar@/^/[rdd]^y & \\
A \ar[r] \ar@/_/[rrd]_x & Q \ar[rd]|{x\vee y} & \\
&& X
}
\end{equation}
where $x$ and $y$ are monomorphisms, $P$ is the pullback of $x$ and $y$,
and $Q$ is the pushout of $a$ and $b$, then the induced map
$\map {x\vee y}QX$ is also a monomorphism.
This follows from \cite{JOHN77}*{Proposition~1.55}.

\item
Monomorphisms are closed under transfinite composition.
This follows from repeated application of \cite{ADRO94}*{Corollary~1.60}.
\end{numlist}

\noindent
From the last three items above, it follows by
\cite{BEKE00}*{Proposition~1.12} that $(\Mono,\rbox\Mono)$ is a
cofibrantly generated weak factorization system.
Now suppose $(\cyl,\gamma,\sigma)$ is a cylinder for $(\Mono,\rbox\Mono)$
and consider the following conditions:
\begin{description}

\item[DH1]
The functor $\cyl$ preserves monomorphisms and all colimits.

\item[DH2]
If $\map fXY$ is a monomorphism then
\begin{equation}
\label{dia:dh2}
\xymatrix{
X \ar[r]^{\gamma^k_X} \ar[d]_f & {\cyl X} \ar[d]^{\cyl f} \\
Y \ar[r]_{\gamma^k_Y} & {\cyl Y}
}
\end{equation}
are pullback squares ($k=0,1$).

\item[DH3]
If $\map fXY$ is a monomorphism then
\begin{equation}
\label{dia:dh3}
\xymatrix{
X+X \ar[r]^-{\gamma_X} \ar[d]_{f+f} & {\cyl X} \ar[d]^{\cyl f} \\
Y+Y \ar[r]_-{\gamma_Y} & {\cyl Y}
}
\end{equation}
is a pullback square.
\end{description}

\noindent
Conditions DH1 and DH2 were introduced by
Cisinski~ \cite{CISI02}*{D\'efinition~2.3}.
We first observe, that it is enough to restrict attention to DH1:

\begin{lemma}
Given a cylinder $(\cyl,\gamma,\sigma)$ for $(\Mono,\rbox\Mono)$,
one has the implications
$\text{DH1} \follows \text{DH2} \follows \text{DH3}$.
\end{lemma}

\begin{proof}
Assume that the cylinder satisfies DH1. For every $\map fXY$, the outer
rectangle in the diagram
$$
\xymatrix{
X \ar[r]^{\gamma^k_X} \ar[d]_f
 & {\cyl X} \ar[r]^{\sigma_X} \ar[d]^{\cyl f} & X \ar[d]^f \\
Y \ar[r]_{\gamma^k_Y} & {\cyl Y} \ar[r]_{\sigma_Y} & Y
}
$$
is always a pullback. If $f$ is a monomorphism then so is $\cyl f$
and hence the left square is also a pullback.
So the cylinder satisfies DH2.

Assume that the cylinder satisfies DH2. Given a monomorphism $\map fXY$,
consider for $k=0,1$ the diagrams
$$
\xymatrix{
X \ar[d]_f \ar[r]^{p^k} & P \ar[d]^h \ar[r]^g & {\cyl X} \ar[d]^{\cyl f} \\
Y \ar[r]_{\iota^k_Y} & Y+Y \ar[r]_{\gamma_Y} & {\cyl Y}
}
$$
where the right square is a pullback and $p^k$ is induced by the maps
$f\iota^k_Y$ and $\gamma^k_X$. By DH2 the outer rectangle is also
pullback and hence the left square is a pullback too.
Because coproducts are universal, the maps $p^0$ and $p^1$ make $P$
into a coproduct of $X$ and $X$.
The canonical isomorphism $\map u{X+X}P$ with $\iota^k_X u=p^k$ then
satisfies $uh=f+f$ and $ug=\gamma_X$. So the cylinder satisfies DH3.
\end{proof}

\begin{corollary}
\label{cor:adm_equal_dh1}
In a Grothendieck topos a cylinder for $(\Mono,\rbox\Mono)$
is cartesian iff it satisfies DH1 (and hence DH2 and DH3) above.
\end{corollary}

\begin{proof}
Let $(\cyl,\gamma,\sigma)$ be a cylinder.

Suppose it is cartesian. Then the left adjoint $\cyl$ preserves all
colimits and we already remarked before that $f\in\Mono$ implies
$\cyl f\in\Mono$. Therefore condition DH1 is satisfied, as well as
conditions DH2 and DH3.

Conversely, suppose that condition DH1 is satisfied.
Now, any locally presentable category is cocomplete (by definition),
co-wellpowered (by \cite{ADRO94}*{Theorem~1.58}) and has a (small)
generator (by \cite{ADRO94}*{Theorem~1.20}).
Therefore it satisfies the dual form of the conditions in Freyd's
Special Adjoint Functor Theorem, and the colimit preserving functor
$\cyl$ is indeed a left adjoint.

To check that $\Mono$ is stable under the $(-)\star\gamma^k$ and
$(-)\star\gamma$, match diagram~\eqref{dia:eff} above with
the diagrams~\eqref{dia:dh2} and \eqref{dia:dh3}.
More precisely, for a monomorphism $\map fXY$ let
$a=f$, $b=\gamma^k_X$, $x=\gamma^k_Y$, $y=\cyl f$ in diagram~\eqref{dia:eff}.
Then $f\star\gamma^k$ coincides (up to isomorphism) with $x\vee y$
and because condition DH2 is satisfied, $x\vee y$ is a monomorphism.
Similarly, conditions DH3 gives that $f\star\gamma$ is a monomorphism.
\end{proof}

We now resume the description of the construction.

\begin{definition}
\label{def:weq}
Let $(\cll,\crr)$ be a weak factorization system, cofibrantly generated
by a subset $I\subseteq\cll$. Let $(\cyl,\gamma,\sigma)$ be a functorial
cylinder and $S\subseteq\cll$ be any subset. Define $\cww(\cyl,S,I)$ as
the class of all those maps $\map fXY$ such that for all objects
$T$ with $(T\to1)\in \rbox{\Lambda(\cyl,S,I)}$ the induced map
$\map{f^\ast}{\ckk(Y,T)/{\htopt}}{\ckk(X,T)/{\htopt}}$
is bijective.
\end{definition}

\begin{remark}
\label{rem:smith1}
Clearly $\cww(\cyl,S,I)$ contains all isomorphisms, has the
2-3 property and is closed under retracts in $\ckk^\two$.
Furthermore, whenever $fg$ and $gf$ lie in $\cww(\cyl,S,I)$, then so do
$f$ and $g$. All these properties follow from the corresponding
properties of bijections. Also note, that for $f\htop g$, one has
$f\in \cww(\cyl,S,I) \iff g\in \cww(\cyl,S,I)$ because the induced maps
$\map{f^\ast,g^\ast}{\ckk(Y,T)/{\htopt}}{\ckk(X,T)/{\htopt}}$ coincide.
\end{remark}

Besides being cofibrantly generated, the weak factorization system
$(\Mono,\rbox\Mono)$ in a Grothendieck topos has the property that
each object is cofibrant, i.e. that each map $(0\to X)$ is in $\cll$.
For convenience, we combine these two properties into one definition:

\begin{definition}
A model structure (weak factorization system) is \df{cofibrant} if it is
cofibrantly generated and every object is cofibrant.
\end{definition}

\begin{lemma}
\label{lem:gamma_is_anext}
Let $(\cll,\crr)$ be a cofibrant weak factorization system,
let $(\cyl,\gamma,\sigma)$ be a cartesian  cylinder and
let $\Lambda:=\Lambda(\cyl,S,I)$ as in Definition~\ref{def:lambda}.
Then the natural maps $\gamma^0$ and $\gamma^1$ have their
components in $\lbox{(\rbox\Lambda)}$.
\end{lemma}

\begin{proof}
Application of Corollary~\ref{cor:starstable} to
$I\star\gamma^k\subseteq\Lambda$ gives 
$\cll\star\gamma^k\subseteq\lbox{(\rbox\Lambda)}$.
Because the left adjoint $\cyl$ must preserve the initial
object, $\gamma^k_X$ differs from $(0\to X)\star\gamma^k$
only by composition with some isomorphism (due to the choice involved
in Definition~\ref{def:star}).
Hence $\gamma^k_X\in\lbox{(\rbox\Lambda)}$.
\end{proof}

We are now ready to state the main result of the section.

\begin{theorem}
\label{thm:3}
Let $\ckk$ be a locally presentable category and $(\cll,\crr)$
a cofibrant weak factorization system generated by a set $I\subseteq\cll$.
Let $(\cyl,\gamma,\sigma)$ be a cartesian cylinder and $S\subseteq\cll$
an arbitrary subset. Then, setting
\begin{align}
\ccc:=\cll &\quad& \cww:=\cww(\cyl,S,I) &\quad&
\cff:=\rbox{(\ccc\intersection\cww)}
\end{align}
gives a cofibrant model structure $(\ccc,\cww,\cff)$ on $\ckk$.
Moreover, $(\cyl,\gamma,\sigma)$ is also a cylinder for this model structure.
\end{theorem}

\begin{remark}
Theorem~\ref{thm:3} does not remain valid if ''cofibrant'' is
weakened to ''cofibrantly generated'' in its statement.
Let $\cgg$ be a (small) generator in $\ckk$ and
consider the set of codiagonal maps
$I:= \{\map{(G|G)}{G+G}G\mid G\in\cgg\}$.

\begin{numlist}
\item
$\rbox I$ is the class $\Mono$ of monomorphisms and
$\lbox{(\rbox I)}$ is the class $\StrEpi$ of strong epimorphisms.

\item
The $(\StrEpi,\Mono)$-factorization of every codiagonal $(X|X)$ as
$$
\xymatrix{X+X\ar[r]^-{(X|X)}&X\ar[r]^X&X}
$$
gives a cylinder $(\cyl,\gamma,\sigma)$ where $\cyl$ and $\sigma$ are
the identity and $\gamma_X=(X|X)$. In particular, $\cyl$ is a left
adjoint and the homotopy relation is equality.

\item
If $\map fXY$ is a strong epimorphism, then $f\star\gamma^0$,
$f\star\gamma^1$ and $f\star\gamma$ are also strong epimorphisms.
This is clear for $\gamma^0$ and $\gamma^1$ because they are
identity transformations.
In the case of $\gamma$, it is enough to observe that
$f=g(f\star\gamma)$, where $g$ is the pushout of $f+f$ along $\gamma_X$.
(Alternatively one can check that $\gamma^\ast\star(-)$ preserves
monomorphisms and apply Lemma~\ref{lem:conjugate}).
\end{numlist}

\noindent
Altogether, $(\StrEpi,\Mono)$ is cofibrantly generated and
$(\cyl,\gamma,\sigma)$ is cartesian.
Going through the construction of
$\Lambda=\Lambda(\emptyset,I)$ in this case, one obtains that
$\Lambda^0$ consists only of isomorphisms and therefore all
$\Lambda^n$ consist only of isomorphisms.
Consequently, every object $X$ satisfies $(X\to1)\in\rbox\Lambda$
and $\cww(\emptyset,I)$ is the class of isomorphisms.
In particular $\rbox\StrEpi$ is not included
in $\cww(\emptyset,I)$.
\end{remark}

The rest of this section will consist of the proof of Theorem~\ref{thm:3}
via Smith's Theorem~\ref{thm:smith}. It turns out that almost
all steps in the proof of \cite{CISI02}*{Th\'eor\`eme~2.13} can be
reused with only minor modifications to verify conditions (1)--(3)
of Theorem~\ref{thm:smith}. However, in verifying condition~(4) we
will depart from \cite{CISI02} and use Part (c) of \ref{thm:heq_is_reductive}
(i.e.~\cite{ROSI07}*{Proposition~3.8}).
Condition (1) already already holds by Remark~\ref{rem:smith1}.
We now turn to condition~(2)

By Lemma~\ref{lem:anodext}, $\rbox{\Lambda(\cyl,S,I)}$ and hence
$\cww(\cyl,S,I)$ do not depend on $I$.
While they do depend on $\cyl$ and $S$ (it will turn out that $S$ is
contained in $\ccc\intersection\cww$ and the components of $\sigma$
lie in $\cww$), the particular choices of $\cyl$ and $S$ do not play
any role in the proof.
Therefore we will simply write $\Lambda$ for $\Lambda(\cyl,S,I)$
and $\cww$ for $\cww(\cyl.S,I)$.
We call an object $X$ \df{fibrant} if $(X\to1)\in\rbox\Lambda$.
In Lemma~\ref{lem:nfib-object_implies_fib} we will show
that these objects coincide with the fibrant objects of the
resulting model structure, so that the terminology is justified.

\begin{definition}[\cite{CISI02}*{D\'efinition~2.15}]
A map $\map fXY$ is a \df{dual strong deformation retract} if there exist
maps  $\map gYX$ and
\newline
$\map h{\cyl{X}}X$ such that the following diagram commutes
\begin{equation}
\label{dia:dsdr}
\xymatrix{
X+X \ar[rr]^{(X|fg)} \ar[d]_{\gamma_X} && X \ar[d]^f
  & Y \ar[l]_g \ar@{=}[dl]\\
{\cyl{X}} \ar[r]_{\sigma_X} \ar[rru]^h& X \ar[r]_f & Y& 
}
\end{equation}
\end{definition}

\begin{lemma}
\label{lem:tfib_implies_dsdr}
Every element of $\rbox{\ccc}$ is a dual strong deformation retract.
\end{lemma}

\begin{proof}
Let $\map fXY \in \rbox{\ccc}$.
Because every object is cofibrant, $f$ is a retraction,
so there is a $\map gYX$ such that the right triangle in
diagram~\eqref{dia:dsdr} commutes.
Because of $(X|fg)f = (f|f) = (X|X)f = \gamma_X\sigma_X f$
the left square of that diagram also commutes. Now $\gamma_X\boxrel f$
gives the desired diagonal $\map h{\cyl{X}}X$.
\end{proof}

\begin{corollary}
\label{cor:smith2}
$\rbox{\ccc}\subseteq\cww$.
\end{corollary}

\begin{proof}
By the previous Lemma, it is enough to check that every
dual strong deformation retract is in $\cww$. If $f$ and
$g$ are as in Diagram~\eqref{dia:dsdr}, then
$X\htop fg$ and $Y = gf$. Using Remark~\ref{rem:smith1}, one
obtains that $fg$ and $gf$ are in $\cww$ and hence $f\in\cww$.
\end{proof}

\begin{remark}
In fact, one has
$\rbox{\ccc}=\rbox{(\ccc\intersection\cww)}\intersection\cww$.
For the direction not covered by the Corollary, factor a
given $f\in \rbox{(\ccc\intersection\cww)}\intersection\cww$
as $f=\ell r$ with $\ell\in \ccc$ and $r \in \rbox\ccc$.
Then $r\in\cww$ and hence $\ell\in\ccc\intersection\cww$.
Therefore $\ell\boxrel f$ and $f$ is a retract of $r$.
So in the language of model structures, the ''trivial fibrations
are indeed those fibrations that are trivial''.
\end{remark}

Condition~(2) holds by Corollary~\ref{cor:smith2}.
Verifying condition~(3) will occupy us until Corollary~\ref{cor:smith3}.

\begin{lemma}
\label{lem:htop_is_trans}
Let $X$ and $T$ be objects with $T$ fibrant.
Then the homotopy relation $\htop$ is an equivalence relation
on $\ckk(X,T)$.
\end{lemma}

\begin{proof}
The relation is clearly reflexive.
For symmetry and transitivity let $u,v,w\in \ckk(X,T)$ and suppose
$v\htop u$ and $v\htop w$ via maps $\map{h,k}{\cyl{X}}X$
with $\gamma_X h=(v|u)$ and $\gamma_X k=(v|w)$.
This gives the solid arrows in the following diagram
$$
\xymatrix{
X+X \ar[r]^{\gamma_X} \ar[d]_{\gamma^0_X+\gamma^0_X}
& {\cyl X} \ar[rr]^{\sigma_X} \ar[d]_p \ar[rd]^{\cyl(\gamma^0_X)} &
& X \ar[dd]^v \\
{\cyl X + \cyl X} \ar[r] \ar@/_/[drrr]_{(h|k)}
& Q \ar[drr]_t \ar[r]^-{\gamma^0_X\star\gamma}
& \relax\cyl\cyl X \ar@{.>}[dr]^d &\\
&&&T
}
$$
where $Q$ is the pushout of $\gamma^0_X+\gamma^0_X$ and $\gamma_X$
and where $t$ is induced by the commuting outer rectangle.
By Lemma~\ref{lem:gamma_is_anext} we have $\gamma^0_X\in\lbox{(\rbox\Lambda)}$.
Applying Corollary~\ref{cor:starstable} to
$\Lambda\star\gamma\subseteq\Lambda$ gives
$\gamma^0_X\star\gamma\in\lbox{(\rbox\Lambda)}$.
Hence $(\gamma^0_X\star\gamma)\boxrel(T\to1)$ and $\map d{\cyl\cyl X}T$ exists.
Therefore the following diagram commutes
$$
\xymatrix{
X+X \ar[d]_{\gamma_X} \ar[r]_{\gamma^1_X+\gamma^1_X}
& {\cyl X + \cyl X} \ar[d]^{\gamma_{\cyl X}} \ar[dr]^{(h|k)} &\\
\relax\cyl X \ar[r]_{\cyl(\gamma^1_X)} & \cyl\cyl K \ar[r]_d & T
}
$$
exhibiting a homotopy from $u$ to $w$.
\end{proof}

\begin{remark}
\label{rem:heq}
With the previous Lemma, the condition for $\map fXY$ to be in $\cww$
can be rephrased in terms of the homotopy relation instead of its
transitive closure: for any given $\map tXT$ with $T$ fibrant there is
a $\map uYT$ with $t\htop fu$ and such a $u$ is determined up to homotopy.
In particular one obtains the following description for maps
between fibrant objects:
\end{remark}

\begin{corollary}
\label{cor:heq}
Suppose $X$ and $Y$ are fibrant. Then  $\map fXY$ is in $\cww$
if and only if there exist a $\map gYX$ with $X\htop fg$ and $Y\htop gf$.
\end{corollary}

\begin{proof}
One direction is clear. If $\map fXY$ is in $\cww$ then using the remark
with $t=\map XXX$ gives a $\map gYX$ with $X\htop fg$.
Therefore $f\htop fgf$ and using the remark with $t=\map fXY$
yields $gf\htop Y$.
\end{proof}

\begin{lemma}
\label{lem:anext_implies_weq}
$\lbox{(\rbox\Lambda)}\subseteq\cww$
\end{lemma}

\begin{proof}
Suppose $\map fXY$ is in $\lbox{(\rbox\Lambda)}$ and 
let $\map tXT$ be a map with $T$ fibrant.
\begin{numlist}
\item
\textsl{Existence:}
Because $f\boxrel (T\to1)$, there exists a $\map uYT$ with $t=fu$, so in
particular $t\htop fu$.

\item
\textsl{Uniqueness:}
Assume $\map{u,v}YT$ with $t\htop fu$ and $t\htop fv$.
By Lemma~\ref{lem:htop_is_trans}, $fu\htop fv$ and there is some
$\map h{\cyl X}X$ with $\gamma_X h=(fu|fv)=(f+f)(u|v)$.
Therefore one has the following diagram
$$
\xymatrix{
X+X \ar[r]^{\gamma_X} \ar[d]_{f+f} & {\cyl X} \ar[d] \ar@/^/[ddr]^h & \\
Y+Y \ar[r] \ar@/_/[drr]_{(u|v)} &
{\pushout{Y+Y}{X+X}{\cyl X}} \ar[dr]|-r & \\
 & & T
}
$$
where $r$ is the induced map from the pushout.
By Corollary~\ref{cor:starstable} $f\star\gamma\in\lbox{(\rbox\Lambda)}$
and hence $(f\star\gamma)\boxrel(T\to1)$, so that $r$ factors through
$f\star\gamma$ via some $\map d{\cyl Y}T$. Therefore
$(u|v) = \gamma_Y d$ and $u\htop v$.\qedhere
\end{numlist}
\end{proof}

\begin{corollary}
\label{cor:sigma_is_weq}
The natural maps $\gamma^0$ and $\gamma^1$ have their components
in $\ccc\intersection\cww$.
The natural map $\sigma$ has its components in $\cww$.
\end{corollary}

\begin{proof}
Let $X$ be any object. Lemma~\ref{lem:gamma_is_anext} and
Lemma~\ref{lem:anext_implies_weq} together give
$\gamma^k_X\in \lbox{(\rbox\Lambda)} \subseteq \ccc\intersection\cww$.
The 2-3 property of $\cww$ then implies $\sigma_X\in\cww$.
\end{proof}

The two implications obtained in Lemma~\ref{lem:tfib_implies_dsdr} and in
Corollary~\ref{cor:smith2} can be strengthened to
equivalences under some conditions.

\begin{lemma}
\label{lem:nfib+dsdr_implies_tfib}
Suppose $f\in\rbox\Lambda$. Then
$$f\in\rbox{\ccc} \iff f \text{\ is a dual strong deformation retract}$$
\end{lemma}

\begin{proof}
The direction ''$\Rightarrow$''  is Lemma~\ref{lem:tfib_implies_dsdr}.
For the direction ''$\Leftarrow$'', assume $\map fXY$ to be
a strong dual deformation retract with maps $\map gYX$ and $\map h{CX}X$
as in diagram~\eqref{dia:dsdr}, i.e.~$gf=X$, $(X|fg)=\gamma_X h$
and $hf=\sigma_X f$ . Any commutative square
$$
\xymatrix{
K \ar[r]^u \ar[d]_c & X \ar[d]^f \\
L \ar[r]_v  & Y
}
$$
with $c\in\ccc$ gives rise to the following diagram
$$
\xymatrix{
& X \ar[r]^{\gamma^1_X} & {\cyl X} \ar[ddr]^h & &\\
K \ar[ru]^u \ar[r]^{\gamma^1_K} \ar[d]_c
& {\cyl K} \ar[d]^p \ar[ru]_{\cyl(u)} &&& \\
L \ar[r]_q \ar[dr]_{\gamma^1_L}
& P \ar[rr]^x &&X \ar[dr]^f &  \\
&{\cyl L} \ar[r]_{\sigma_L} & L \ar[r]_v & Y \ar[u]^g \ar@{=}[r] & Y
}
$$
where $P$ is the pushout of $c$ and $\gamma^1_K$ and $\map xPX$ is
induced by $\gamma^1_K \cyl(u) h = u\gamma^1_X h = ufg = cvg$.
Testing against $p$ and $q$ gives the commutativity of the lower right
square in
$$
\xymatrix{
K \ar[r]^u \ar[rd]^{\gamma^0_K} \ar[dd]_c
& X \ar[r]^{\gamma^0_X} & {\cyl X} \ar[drr]^h && \\
& {\cyl K} \ar[r]_p \ar[ru]_{\cyl(u)}
& P \ar[rr]^x \ar[d]|-{c\star\gamma^1} && X \ar[d]^f \\
L \ar[rr]_{\gamma^0_L} && {\cyl L} \ar[r]_{\sigma_L} \ar@{.>}[rru]^d
& L \ar[r]_v & Y
}
$$
and hence $(c\star\gamma^1)\boxrel f$ gives a diagonal $\map d{\cyl L}X$.
The outer diagram then shows that $d':=\map{\gamma^0_L d}LY$
is the desired diagonal.
\end{proof}

\begin{lemma}
\label{lem:nfib+weq_implies_tfib}
Suppose $f\in\rbox\Lambda$ with fibrant codomain.
Then
$$f\in\rbox{\ccc} \iff f\in\cww$$
\end{lemma}

\begin{proof}
The direction ''$\Rightarrow$''  is Corollary~\ref{cor:smith2}.
For the direction ''$\Leftarrow$'', assume $\map fXY \in\cww$
and $Y$ fibrant.
By Lemma~\ref{lem:nfib+dsdr_implies_tfib}, it is sufficient to show that $f$
is a dual strong deformation retract. We will construct
$\map gYX$ and $\map h{\cyl X}X$, such that the equations
in diagram~\eqref{dia:dsdr} are satisfied.

Because $f$ and $(Y\to1)$ are in $\rbox\Lambda$, the same holds
for $(X\to1)$. By Corollary~\ref{cor:heq} there exists a $\map gYX$
with $X\htop fg$ and $Y\htop gf$. Let $\map k{\cyl X}X$ be the
homotopy from $X$ to $fg$.
\begin{numlist}
\item
\textsl{One may assume $Y=gf$.}
Consider the following diagram
$$
\xymatrix{
  & Y \ar[r]^g \ar[d]_{\gamma^1_Y} & X \ar[d]^f \\
Y \ar[r]^{\gamma^0_Y} & {\cyl Y} \ar[r] \ar@{.>}[ru]^d & Y
}
$$
where the right square comes from $Y\htop gf$.
The diagonal $\map d{\cyl Y}X$ exists because
$\gamma^1_Y\in\lbox{(\rbox\Lambda)}$ by Lemma~\ref{lem:gamma_is_anext}.
Let $g':=\gamma^0_Yd$.
Then $g'f=Y$ and $(g'|g)=\gamma_Yd$.
Hence $X\htop fg\htop fg'$ and by Lemma~\ref{lem:htop_is_trans}
we have $X\htop fg'$ via some homotopy $k'$.
Now replace $g$ and $k$ by $g'$ and $k'$.

\item
\textsl{There are maps
$\map x{\pushout{\cyl X+\cyl X}{X+X}{\cyl X}}X$
and
$\map d{\cyl\cyl X}X$
such that the following diagram commutes:}
$$
\xymatrix{
{\pushout{\cyl X+\cyl X}{X+X}{\cyl X}}
 \ar[rrr]^x \ar[d]_{\gamma^1_X\star\gamma_X} &&&X\ar[d]^f \\
{\cyl\cyl X} \ar[r]_{\sigma_{\cyl X}} \ar[rrru]^d
& {\cyl X} \ar[r]_k & X \ar[r]_f & Y
}\eqno{(*)}
$$
The equation
\begin{align*}
(\gamma^1_X+\gamma^1_X)(k|kfg)
&= (\gamma^1_Xk|\gamma^1_Xkfg)
= (fg|fgfg)\\
= (X|X)fg 
&= \gamma_X\sigma_X fg
\end{align*}
induces $x$ in the following diagram
$$
\xymatrix{
X+X \ar[r]^-{\gamma_X} \ar[d]_{\gamma^1_X+\gamma^1_X}
& {\cyl X} \ar[r]^{\sigma_X} \ar[d]_j & X \ar[d]^f \\
{\cyl X+\cyl X} \ar[r]^-i \ar@/_/[rrd]_{(k|kfg)}
& Q \ar[rd]^x & Y \ar[d]^g \\
&&X
}
$$
where $Q$ is the pushout of $\gamma^1_X+\gamma^1_X$ and $\gamma_X$
with coprojections $\map i{\cyl X+\cyl X}Q$ and $\map j{\cyl X}Q$.
The commutativity of the outer rectangle of diagram~($*$)
now follows from the following two equations
\begin{align*}
ixf
&= (k|kfg)f = (kf|kf)
= (\cyl X|\cyl X)kf
= \gamma_{\cyl X} \sigma_{\cyl X} kf \\
&= i (\gamma^1_X\star\gamma_X) \sigma_{\cyl X} kf \\
& \\
jxf
&= \sigma_X fgf = \sigma_X \gamma^1_X kf
= \cyl(\gamma^1_X)\sigma_{\cyl X} kf \\
&= j (\gamma^1_X\star\gamma_X) \sigma_{\cyl X} kf 
\end{align*}
Finally the existence of the diagonal $d$ in diagram~($*$)
follows from $\gamma^1_X\star\gamma_X\in\lbox{(\rbox\Lambda)}$.

\item
\textsl{With $x$ and $d$ as in (2), let
$h:=\map{\cyl(\gamma^0_X)d}{\cyl X}X$.
Then the following diagram commutes:}
$$
\xymatrix{
X+X \ar[rr]^{(X|fg)} \ar[d]_{\gamma_X} && X \ar[d]^f \\
{\cyl{X}} \ar[r]_{\sigma_X} \ar[rru]^h& X \ar[r]_f & Y
}
$$
The lower triangle is the equation
$$
\cyl(\gamma^0_X)df = \cyl(\gamma^0_X)\sigma_{\cyl X} kf
= \sigma_X \gamma^0_X kf = \sigma_X f
$$
The upper triangle is the equation
\begin{align*}
\gamma_X \cyl(\gamma^0_X)d
&= (\gamma^0_X+\gamma^0_X)\gamma_{\cyl X} d \\
&= (\gamma^0_X+\gamma^0_X) i (\gamma^1_X\star\gamma_X)d \\
&= (\gamma^0_X+\gamma^0_X) i x \\
&= (\gamma^0_X+\gamma^0_X) (k|kfg) \\
&= (\gamma^0_Xk|\gamma^0_Xkfg) \\
&= (X|fg)
\end{align*}
\end{numlist}
Altogether, $h$ and $g$ satisfy the equations
in diagram~\eqref{dia:dsdr}.
\end{proof}

\begin{corollary}
\label{cor:cof+fibcod+weq_implies_anext}
Let $\map fXY\in \ccc$ with fibrant codomain.
Then
$$
f\in \cww \iff f\in \lbox{(\rbox\Lambda)}
$$
\end{corollary}

\begin{proof}
The direction ''$\Leftarrow$'' is Lemma~\ref{lem:anext_implies_weq}.
For the direction ''$\Rightarrow$'', suppose $f\in\cww$.
Factor $f$ as $ip$ with $i\in\lbox{(\rbox\Lambda)}$
and $p\in\rbox\Lambda$.
Then $p$ satisfies the condition of the previous Lemma and
hence
$$
f\in \cww \iff p\in \cww \iff p\in \rbox\ccc
$$
so that in particular $f\boxrel p$. Therefore $f$ is a retract
of $i$ and lies in $\lbox{(\rbox\Lambda)}$.
\end{proof}

\begin{lemma}
\label{lem:nfib-object_implies_fib}
Let $\cnn = \{ p\in\rbox\Lambda\mid p \text{ has a fibrant codomain} \}$.
Then
\begin{alphlist}
\item
$\ccc\intersection\cww = \ccc\intersection\lbox\cnn$.

\item
$\cnn \subseteq \rbox{(\ccc\intersection\cww)}$

\item
$(X\to1)\in\rbox\Lambda \iff (X\to1)\in \rbox{(\ccc\intersection\cww)}$
\end{alphlist}
\end{lemma}

\begin{proof}
First observe that because of
$\lbox{(\rbox\Lambda)}\subseteq\ccc\intersection\cww$
(Lemma~\ref{lem:anext_implies_weq} together with condition (b) of
Definition~\ref{def:cart})
we have $\rbox\Lambda\supseteq\rbox{(\ccc\intersection\cww)}$
and hence the implication ''$\Leftarrow$'' in (c) always holds.
The implication ''$\Rightarrow$'' in (c) follows from (b).
Moreover, (a) implies (b) via
$\ccc\intersection\lbox\cnn\subseteq\lbox\cnn$.
So it is enough to show (a).
Let $\map cKL$ be any map in $\ccc$.
Factor $(L\to1)$ through some $\map uL{L'}$ with
$u\in\lbox{(\rbox\Lambda)}$ and $L'$ fibrant.
Then in particular $u\in\ccc$ with fibrant codomain and
hence $u\in\cww$ by Corollary~\ref{cor:cof+fibcod+weq_implies_anext}.
Therefore
$$
c\in\cww \iff cu\in\cww \iff cu\in\lbox{(\rbox\Lambda)}
\eqno{(*)}
$$
where the second equivalence again results from
Corollary~\ref{cor:cof+fibcod+weq_implies_anext}.

\begin{numlist}
\item
Suppose $c\in\cww$.
Consider any $p\in\cnn$ and maps $\map xKX$ and $\map yLY$ as in
the following diagram:
$$
\xymatrix{
K \ar[r]^c \ar[d]_x & L \ar[r]^u \ar[d]^y & 
L' \ar@{.>}@/_/[dll]_(0.7)d \ar@{.>}[ld]^{y'} \\
X \ar[r]_p & Y & 
}
$$
Then $\map {y'}{L'}Y$ exists because $u\boxrel(Y\to1)$
and $\map d{L'}X$ exists because of the above ($*$).
The equations $cud=x$ and  $udp=uy'=y$ then exhibit $\map {ud}LX$
as the desired diagonal.

\item
Suppose $c\in\lbox\cnn$.
Factor $cu$ as $cu=xp$ with $x\in\lbox{(\rbox\Lambda)}$ and
$p\in\rbox\Lambda$. Because $u$ has fibrant codomain, the
same holds for $p$ and hence $p\in\cnn$.
Because $u\in\lbox{(\rbox\Lambda)}\subseteq\lbox\cnn$, also
$cu\in\lbox\cnn$. Therefore $cu$ is a retract of $p$ and
hence $cu\in\lbox{(\rbox\Lambda)}\subseteq \cww$.
Now by ($*$) above, $c\in\cww$.\qedhere
\end{numlist}
\end{proof}

\begin{corollary}
\label{cor:smith3}
$\ccc\intersection\cww$ is stable under pushouts, transfinite composition
and retracts.
\end{corollary}

\begin{proof}
By part (a) of the previous Lemma, $\ccc\intersection\cww$ can
be expressed as the intersection of two classes, each of which
is stable under these operations.
\end{proof}

It now remains to verify condition~(4).
We want to express $\cww$ as the full preimage (under some accessible functor)
of the class of homotopy equivalences with respect to some final cylinder.
Observe that the cylinder used in the construction may not be final.

\begin{lemma}
There is a final refinement $(\cyl',\gamma',\sigma')$ of
$(\cyl,\gamma,\sigma)$ such that for any two maps $\map{f,g}XY$
with fibrant codomain we have
$$ f\htop g \pmod{\cyl'} \iff  f\htop g \pmod\cyl $$
In particular, the two cylinders agree on the notion of homotopy
equivalences between fibrant objects.
\end{lemma}

\begin{proof}
Let $\sigma=\lambda\rho$ be a functorial $(\ccc,\rbox\ccc)$-factorization of
$\sigma$ and for each object $X$ set $\cyl' X = \cod(\lambda_X)$,
$\gamma'_X =\gamma_X\lambda_X$ and $\sigma'_X = \rho_X$.
Then $(\cyl',\gamma',\sigma')$ is a final refinement of 
$(\cyl,\gamma,\sigma)$ and the direction ''$\Rightarrow$'' was already
noted in part (d) of Observation~\ref{obs:cylinders}.

Now assume $f\htop g \pmod\cyl$ for maps $\map{f,g}XY$ with $Y$ cofibrant.
Let $\map h{\cyl X}Y$ be a homotopy from $f$ to $g$ and consider the
square:
$$
\xymatrix{
{\cyl X}  \ar[r]^h \ar[d]_{\lambda_X} & Y \ar[d] \\
{\cyl' X} \ar[r] & 1
}
$$
Corollary~\ref{cor:smith2} gives $\rho_X\in\rbox\ccc\subseteq\cww$ and
Corollary~\ref{cor:sigma_is_weq} gives $\lambda_X\rho_X=\sigma_X\in\cww$.
Therefore the 2-3 property of $\cww$ forces $\lambda_X\in\cww$ and
hence $\lambda_X\in\ccc\intersection\cww$.
By part (c) of Lemma~\ref{lem:nfib-object_implies_fib} we have
$(Y\to1)\in\rbox{(\ccc\intersection\cww)}$.
This gives the desired diagonal $\map d{\cyl' X}Y$ of the above square,
establishing $f\htop g \pmod{\cyl'}$.
\end{proof}

\begin{corollary}
\label{cor:smith4}
The class $\cww$ satisfies the solution set condition.
\end{corollary}

\begin{proof}
By Lemma~\ref{lem:reductive_implies_solution-set}, it is sufficient
to exhibit $\cww$ as the full image of some accessible functor.
Let $\map L{\ckk}{\ckk}$ be the fibrant replacement functor given by
the weak factorization system $(\lbox{(\rbox\Lambda)},\rbox\Lambda)$,
which is accessible by part (b) of Theorem~\ref{thm:funcfact_is_accessible}.
Via composition, $L$ induces a functor
$\map{L_\ast}{\ckk^\two}{\ckk^\two}$,
which is also accessible because colimits in $\ckk^\two$ are
calculated pointwise.

Let $\map fXY$ be any map.
\begin{numlist}
\item
\textsl{$f\in\cww\iff Lf\in\cww$}

Consider the square
$$
\xymatrix{
X \ar[r]^{\ell_X} \ar[d]_f & LX \ar[d]^{Lf} \\
Y \ar[r]^{\ell_Y} & LY 
}
$$
where $\ell_X,\ell_Y\in\lbox{(\rbox\Lambda)}$ are given by the
functorial factorization. By Lemma~\ref{lem:anext_implies_weq}
$\ell_X$ and $\ell_X$ lie in $\cww$. Now the 2-3 property
of $\cww$ gives the above equivalence.

\item
\textsl{$Lf\in\cww \iff Lf \text{ is a homotopy equivalence} \pmod\cyl$}

By construction, $Lf$ has fibrant domain and codomain.
The equivalence now follows from Corollary~\ref{cor:heq}.
\end{numlist}
Let $(\cyl',\gamma',\sigma')$ be a final refinement of
$(\cyl,\gamma,\sigma)$ as in the previous Lemma.
Then point (2) still remains valid with $\cyl'$ in place of $\cyl$.
Therefore $\cww$ is the preimage, under the accessible functor $L_*$,
of the class of homotopy equivalences determined by $\cyl'$.
By part (c) of Theorem~\ref{thm:heq_is_reductive} that class
is the full image of an accessible functor. It is also isomorphism-closed.
Hence the same holds for $\cww$ by
Lemma~\ref{lem:preimage_of_replete_reductive}.
\end{proof}

\begin{proof}[Proof of Theorem~\ref{thm:3}]
By Remark~\ref{rem:smith1}, Corollary~\ref{cor:smith2},
Corollary~\ref{cor:smith3} and Lemma~\ref{cor:smith4},
the classes $\ccc$ and $\cww$ satisfy the conditions of
Smith's Theorem \ref{thm:smith}.
\end{proof}

\section{Left determination}

\nobreak\noindent
Let $\ckk$ be any complete and cocomplete category.
Given a fixed class $\ccc$ of maps in $\ckk$, consider the following
conditions on a class $\cww$ of maps:

\begin{romlist}

\item
$\cww$ has the 2-3 property.

\item
$\cww$ is closed under retracts in $\ckk^\two$.

\item
$\rbox\ccc \subseteq \cww$.

\item
$\ccc \intersection \cww$ is closed under pushouts and transfinite
composition.
\end{romlist}

\noindent
Then each condition is stable under intersections, i.e.~if it
is satisfied by every $\cww_i$ in some (possibly large)
family $\cww_i$ ($i\in I$), then it is also satisfied by their intersection.
Also, whenever $\ccc$ and $\cww$ are part of a model structure
$(\ccc,\cww,\cff)$, then $\cww$ satisfies all of the above conditions.

The following Definition was given by Cisinski~\cite{CISI02}*{D\'efinition~3.4}
for the special case where $\ckk$ is a (Grothendieck) topos and
$\ccc$ is the class of all monomorphisms.

\begin{definition}
\label{def:basic_loc}
Let $\ccc$ be a fixed class of maps in $\ckk$.
The class $\cww$ is a \df{localizer} for $\ccc$ if it satisfies
conditions (i),(iii) and (iv) above. For any given class $S$ of
maps, $\cww(S)$ denotes the smallest localizer containing $S$.
In particular $\cww(\emptyset)$ is the smallest localizer.
\end{definition}

The following Definition was given by
Rosick\'y and Tholen~\cite{ROTH03}*{Definition~2.1}.

\begin{definition}
\label{def:left_determined}
Given a class $\ccc$ of maps in $\ckk$, write $\cww_\ccc$
for the smallest class satisfying conditions (i)--(iv) above.
A model structure $(\ccc,\cww,\cff)$ is \df{left determined}
if $\cww=\cww_\ccc$.
\end{definition}

\begin{remark}
One always has $\cww(\emptyset)\subseteq\cww_\ccc$ and also
$\cww_\ccc\subseteq\cww$ for any model structure $(\ccc,\cww,\cff)$.
As in Definition~\ref{def:basic_loc}, one can also consider
the smallest class $\cww_{S,\ccc}$ of maps satisfying conditions
(i)--(iv) and containing a class of maps $S$.
Then $\cww(S)\subseteq\cww_{S,\ccc}$ and
$\cww_{S,\ccc}\subseteq\cww$ for any model structure $(\ccc,\cww,\cff)$
satisfying $S\subseteq\cww$.
In particular, whenever $\ccc$ and $\cww(S)$ give a model structure
then $\cww(S)=\cww_{S,\ccc}$.
\end{remark}

We now return to the situation of the previous section, so we
assume from now on that $\ckk$ is locally presentable.
The following Lemma and Theorem are adapted from
\cite{CISI02}*{Proposition~3.8} and \cite{CISI02}*{Th\'eor\`eme~3.9}.

\begin{lemma}
Let $(\ccc,\rbox\ccc)$ be a cofibrant weak factorization system
in $\ckk$, generated by a subset $I\subseteq\ccc$.
Let $(\cyl,\gamma,\sigma)$ be a cartesian cylinder and let
$S\subseteq\ccc$ be a set of maps.
Then $\cww(\cyl,S,I) = \cww(\Lambda(\cyl,S,I))$.
\end{lemma}

\begin{proof}
We will again write $\Lambda$ for $\Lambda(\cyl,S,I)$
and $\cww$ for $\cww(\cyl,S,I)$.
The inclusion $\cww(\Lambda)\subseteq \cww$ holds
because $\Lambda\subseteq \cww$ by Lemma~\ref{lem:anext_implies_weq}.

Now given any $\map fXY \in\cww$,
use $(\cell(\Lambda),\rbox\Lambda)$-factorizations of $(X\to1)$
and $(Y\to1)$ to obtain a diagram
$$
\xymatrix{
X \ar[r]^{\ell_X} \ar[dd]_f & X' \ar[dd]_{f'} \ar[rd]^z & \\
&& Z \ar[dl]^y \\
Y \ar[r]_{\ell_Y} & Y' & 
}
$$
where $\ell_X$ and $\ell_Y$ are in $\cell(\Lambda)$, $X'$ and $Y'$ are
fibrant, $f'$ is induced by this factorization and $f'=zy$ is in turn
a factorization with $z\in\cell(\Lambda)$ and $y\in\rbox\Lambda$.
In particular $\ell_X$, $\ell_Y$ and $z$ are in $\cww(\Lambda)$.
Then the 2-3 property gives
$$
f\in\cww
\follows
y\in\cww
\iff
y\in\rbox\ccc
\follows
y\in\cww(\Lambda)
\follows
f\in\cww(\Lambda)
$$
where the equivalence in the middle is given by
Lemma~\ref{lem:nfib+weq_implies_tfib}
\end{proof}

\begin{theorem}
Let $(\ccc,\rbox\ccc)$ be a cofibrant weak factorization system in $\ckk$
and $S$ be an arbitrary set of maps (not necessarily included in $\ccc$).
Suppose that $(\cyl,\gamma,\sigma)$ is a cartesian cylinder such that
all components of $\sigma$ lie in $\cww(S)$.
Then, setting $\cww:=\cww(S)$ and $\cff:=\rbox{(\ccc\intersection\cww)}$
gives a cofibrant model structure $(\ccc,\cww,\cff)$ on $\ckk$.
Also $\cww(S)=\cww_{S,\ccc}$.
\end{theorem}

\begin{proof}
First observe, that one may assume $S\subseteq\ccc$:
factor each $s\in S$ as $s=c_sr_s$ with $c_s\in\ccc$ and $r_s\in\rbox\ccc$
and consider $S':=\{c_s\mid s\in S\}$. Any given localizer contains $S$
if and only if it contains $S'$, because all the $r_s$ lie in it.
Therefore $\cww(S')=\cww(S)$.

Now assume $S\subseteq\ccc$. Let $I$ be some generating subset of $\ccc$.
By the previous Lemma, it is enough to show $\cww(\Lambda(\cyl,S,I))=\cww(S)$.
We will write $\Lambda(S)$ for $\Lambda(\cyl,S,I)$.

The inclusion $S\subseteq\Lambda(S)$ already forces
$\cww(S) \subseteq \cww(\Lambda(S))$ and therefore it
remains to show $\Lambda(S)\subseteq\cww(S)$.

By assumption, the components of $\sigma$ lie in $\cww(S)$.
Consequently the components of $\gamma^0$ and $\gamma^1$ lie
in $\ccc\intersection\cww(S)$.
We will now show $\Lambda^n(S) \subseteq \cww(S)$ for all $n\geq0$.

\begin{numlist}
\item
We already have $S\subseteq\cww(S)$.
Let $\map fXY$ be in $\ccc$ and consider the following diagram
used for the definition of $f\star\gamma^0$
$$
\xymatrix{
X \ar[d]_f \ar[r]^{\gamma^0_X} & {\cyl X} \ar[d]^p \ar@/^/[ddr]^{\cyl f}& \\
Y \ar[r]_q \ar@/_/[drr]_{\gamma^0_Y} & Q \ar[dr]|{f\star\gamma^0}& \\
&& {\cyl Y}
}
$$
where $Q$ is the pushout of $f$ and $\gamma^0_X$.
Because $\gamma^0_X\in \ccc\intersection\cww(S)$ we have $q\in\cww(S)$.
Together with $\gamma^0_Y\in\cww(S)$ this
gives $f\star\gamma^0\in\cww(S)$.
In the same way $f\star\gamma^1\in\cww(S)$.
Hence $I\star\gamma^0$ and $I\star\gamma^1$ are contained in $\cww(S)$

\smallskip
\item
Assume $\Lambda^n(S)\subseteq\cww(S)$ and let
$\map fXY$ be in $\Lambda^n(S)$.
By assumption $f\in\cww(S)$ and hence $f$ lies in $\ccc\intersection\cww(S)$.
Then the same holds for $f+X$ and $Y+f$ (being pushouts of $f$), as for
their composition $f+f=(f+X)(Y+f)$.
Moreover $f\in\ccc\intersection\cww(S)$ together with
$\gamma^0_X,\gamma^0_Y\in\cww(S)$ force $\cyl f\in\cww(S)$
by the 2-3 property.
Altogether, in the following diagram used for the definition of
$f\star\gamma$
$$
\xymatrix{
X+X \ar[d]_-{f+f} \ar[r]^{\gamma_X} & {\cyl X} \ar[d]^r \ar@/^/[ddr]^{\cyl f}& \\
Y+Y \ar[r] \ar@/_/[drr]_{\gamma_Y} & Q \ar[dr]|{f\star\gamma}& \\
&& {\cyl Y}
}
$$
both maps $r$ and $\cyl f$ lie in $\cww$, and hence $f\star\gamma\in\cww$.
\qedhere
\end{numlist}
\end{proof}

In view of Corollary~\ref{cor:sigma_is_weq} it is clear that the
condition of $\sigma$ having its components in $\cww(S)$ cannot be
omitted from the Theorem. This condition will always be satisfied
(regardless of the $\cww(S)$ in question) whenever the cylinder is final,
i.e.~when $\sigma$ has its components in $\rbox\ccc$.

\begin{corollary}
\label{cor:final_gives_minimal}
Let $(\ccc,\rbox\ccc)$ be a cofibrant weak factorization system in $\ckk$
and suppose that there is a final cartesian cylinder for $(\ccc,\rbox\ccc)$.
Then $\ccc$, $\cww(S)$ and $\rbox{(\ccc\intersection\cww(S))}$ form a
cofibrantly generated model structure. In particular for $S=\emptyset$,
the construction of Theorem~\ref{thm:3} gives a left determined
model structure.
\end{corollary}

\begin{remark}
\label{rem:cylfunc}
The above result also shows, that the construction of the model structure
from $(\ccc,\rbox{\ccc})$ and $S$ does not depend on the choice of the
cylinder used. For example, if the underlying category is distributive
and if the class $\ccc$ is stable under pullbacks along product projections,
then any factorization of the codiagonal $\map{(1|1)}{2={1+1}}1$ as a
composition of some $\map g2V$ and $\map sV1$ with $g\in\ccc$ and
$s\in\rbox\ccc$ will provide a final cylinder with $\cyl=(-)\times V$,
$\gamma=(-)\times g$ and $\sigma=(-)\times s$.
If $V$ is exponentiable then $\cyl$ is a left adjoint.
\end{remark}

\begin{example}
\label{ex:topos}
Let $\map\top1\Omega$ be the subobject classifier of a Grothendieck
topos $\cee$ and let $\map\bot1\Omega$ be the characteristic map of $0\to1$,
which means that $\bot$ is the uniquely determined map in the
pullback:
$$
\xymatrix{
0 \ar[d] \ar[r] & 1 \ar[d]^\top \\
1 \ar[r]_\bot & \Omega
}
$$
Then the induced map $\map{(\bot|\top)}{1+1}\Omega$ is a monomorphism
(this is just another instance of Diagram~\eqref{dia:eff}).
Since $\Omega$ is injective, this gives a $(\Mono,\rbox\Mono)$-factorization
of the codiagonal $\map{(1|1)}{1+1}1$.
\newline
Therefore $(-)\times\Omega$ gives a final cylinder and the natural map
$\gamma$ is given as $(-)\times(\bot,\top)$.

Because $\cee$ is cartesian closed, $(-)\times\Omega$ is a left adjoint
and it clearly preserves monomorphisms.
By Corollary~\ref{cor:adm_equal_dh1}, the resulting cylinder is cartesian.
\end{example}

\section{Examples}

\nobreak\noindent
In this section, we will examine examples, where the underlying categories
are locally presentable, but not toposes. However, except for the
last one, they are still cartesian closed and cylinders can be obtained
from suitable factorizations of the codiagonals $2\to1$
as indicated in Remark~\ref{rem:cylfunc}.

Moreover, the homotopy relation is already determined by $\cyl(1)$
in the sense that two maps $\map{f,g}XY$ are homotopic if and only if
their exponential adjoints
$\map{\ulcorner{f}\urcorner,\ulcorner{g}\urcorner}1{Y^X}$
are homotopic. This latter condition often has a direct description in terms
of the structure of $Y^X$, so that it is sufficient to know when two
elements $\map{x,y}1X$ are homotopic.

The first example also provides an instance of the second line of
generalization, in that the class of cofibrations is not the class
of monomorphisms.

\begin{example}
\label{ex:cat}
Consider $\ckk=\Cat$, the category of small categories and functors.
It has a model structure, the so called ''folk model structure'',
where the cofibrations are those functors that are injective on
objects, and the weak equivalences are the usual categorical equivalences.
This model structure has been known for some time (hence the name),
the first published source seems to be Joyal and Tierney~\cite{JOTI90}.
It has also been later reproved and described in detail by
Rezk~\cite{REZK96}. We will show that this model structure
is left determined by rebuilding it from a generating set of
cofibrations and a final cartesian cylinder.

Recall that for any set $S$ one has the discrete category on
its elements (written also as $S$) and the indiscrete category
(i.e. the connected groupoid with trivial object groups) on its elements,
which we will write as $\overline{\underline S}$.
These two constructions give functors in the obvious way to provide
left and right adjoints for the underlying object functor $\map \Ob\Cat\Set$.
In particular we write $2$ and $\overline{\underline2}$ for the discrete
and the indiscrete category on two objects.
Moreover, we write $\underline2$ for the linearly ordered set $\{0,1\}$
and $P$ for the ''parallel pair'', i.e. the pushout of
the inclusion $2\inclusion\underline2$ with itself.

Consider
$I=\{(0\inclusion1), (2\inclusion\underline2), \map pP{\underline2}\}$,
where the last functor maps both nontrivial arrows of $P$ to the
nontrivial arrow of $\underline2$.
\begin{numlist}

\item
We first check that $I$ is a set of generating cofibrations.
Clearly $\rbox I$ consists of all those functors, which are full,
faithful and surjective on objects. Moreover, for any map $f$ one has
\newline
$f\in\lbox{(\rbox I)} \iff \Ob(f) \text{ is a monomorphism}$.

\noindent
For the direction ''$\Rightarrow$'', observe that the functor
$(\overline{\underline2}\to1)$ is in $\rbox I$ and that
$f\boxrel (\overline{\underline2}\to1)$ forces
$\Ob(f) \boxrel (2\to1)$ in $\Set$.

\noindent
Conversely, consider a square
$$
\xymatrix{
A \ar[d]_i \ar[r]^f & X \ar[d]^p \\
B \ar[r]_g & Y
}
$$
where $p\in\rbox{I}$ and $i$ is injective on objects.
Define $\map hBX$ on objects by $h(i(a)) = f(a)$ and
$h(b)\in p^{-1}(g(b))$ for $b\notin i(A)$. This can be done
because $\Ob(i)$ is injective and $\Ob(p)$ is surjective.
For a morphism $\map ub{b'}$ in $B$, define $\map {h(u)}{h(b)}{h(b')}$
to be the unique element of $X(h(b),h(b'))\intersection p^{-1}(g(u))$.
This works because $p$ is full and faithful.
Then $h$ is the desired diagonal.

\item
The cylinder functor $\cyl=(-)\times\overline{\underline2}$
is obtained from the factorization $2\inclusion\overline{\underline2}\to1$
and $\map{\gamma_X}{X\times 2}{X\times\overline{\underline2}}$
is the usual inclusion. Because $(\overline{\underline2}\to1)$ is in
$\rbox I$, the resulting cylinder is final. Two objects $\map{x,y}1X$
of a category $X$ are homotopic iff they are isomorphic.
Therefore two functors $\map{f,g}XY$ are homotopic iff they
are naturally isomorphic.

\item
It remains to check condition (b) of Definition~\ref{def:cart}, i.e.~stability
of $I$ under $(-)\star\gamma$ and $(-)\star\gamma^k$.

For the case of $\gamma$, consider a diagram
$$
\xymatrix{
X+X \ar[r]^-{\gamma_X} \ar[d]_{f+f} & {\cyl X} \ar[d] \ar@/^/[dr]^{\cyl f} & \\
Y+Y \ar@/_10pt/[rr]_{\gamma_Y} \ar[r]^-q & Q \ar[r]^{f\star\gamma} & {\cyl Y}
}
$$
where $Q$ is a pushout of $f+f$ and $\gamma_X$.
The maps $\Ob(\gamma_X)$ and $\Ob(\gamma_Y)$ are bijective.
Because the functor $\Ob$ preserves pushouts, the map $\Ob(q)$ is also
bijective and hence $\Ob(f\star\gamma)$ is bijective.

For the case of $\gamma^0$ and $\gamma^1$ one can calculate directly that
the following two diagrams
$$
\xymatrix{
2 \ar[r]^-{\gamma^k_2} \ar[d]_{\gamma_1}
 & 2\times\overline{\underline2}
   \ar[d]^{{\gamma_1}\times\overline{\underline2}} \\
{\underline2} \ar[r]_-{\gamma^k_{\underline2}}
 & {\underline2}\times\overline{\underline2}
}
\qquad
\xymatrix{
P \ar[r]^-{\gamma^k_P} \ar[d]_p
 & P\times\overline{\underline2} \ar[d]^{p\times\overline{\underline2}} \\
{\underline2} \ar[r]_-{\gamma^k_{\underline2}}
 & {\underline2}\times\overline{\underline2}
}
$$
are pushout squares and hence $(2\inclusion\underline2)\star\gamma^k$ and
$p\star\gamma^i$ are isomorphisms.
Moreover, $(0\to1)\star\gamma^k=\gamma^k_1$.

\item
Now for the computation of $\Lambda(\emptyset,I)$.
By (3) above, $\Lambda^0(\emptyset,I)$ consists
of isomorphisms and the two inclusions
$\map {\gamma^0_1,\gamma^1_1}1{\overline{\underline2}}$.
A direct computation gives that
$$
\xymatrix{
1+1 \ar[r]^{\gamma_1} \ar[d]_{\gamma^k_1+\gamma^k_1}
 & 1\times\overline{\underline2}
   \ar[d]^{\gamma^k_1\times\overline{\underline2}} \\
{\overline{\underline2}}+{\overline{\underline2}}
 \ar[r]_-{\gamma_{\overline{\underline2}}}
 & {\overline{\underline2}}\times{\overline{\underline2}}
}
$$
is a pushout square and hence $\gamma^k_1\star\gamma$ is an isomorphism.
Therefore
$\rbox{\Lambda(\emptyset,I)}=\rbox{\Lambda^0(\emptyset,I)}
=\rbox{\{\gamma_1^0,\gamma_1^1\}}$
and every object of $\Cat$ is fibrant.

\item
From Corollary~\ref{cor:heq} we obtain that $\cww=\cww(\emptyset,I)$
consists of the categorical equivalences, which completes the
construction.
\end{numlist}
\end{example}

\noindent
The following Lemma gives a method to build new examples from
old ones by inducing Cisinski's construction on certain
reflective subcategories.

\begin{lemma}
\label{lem:induced}
Let $\ckk$ be a locally presentable category with a cofibrant weak
factorization system generated by a set $I$, a cylinder
$(\cyl,\gamma,\sigma)$ and a reflection $\map R\ckk\caa$  onto
a full subcategory $\caa$ which is also locally presentable.
Then the restriction of $\map{R\cyl}\ckk\caa$ to $\caa$ provides
a cylinder $(R\cyl,R\gamma,R\sigma)$ for the cofibrant weak
factorization system generated by $RI$ in $\caa$.
Moreover, the following holds:
\begin{alphlist}
\item
The two cylinders $(\cyl,\gamma,\sigma)$ and $(R\cyl,R\gamma,R\sigma)$
determine the same homotopy relation on $\caa$.

\item
For any $S\subseteq \lbox{(\rbox I)}$ one has
$\Lambda(R\cyl,RS,RI) = R\Lambda(\cyl,S,I)$.
Consequently $\Lambda(R\cyl,RS,RI)$ and $\Lambda(\cyl,S,I)$ determine
the same class of fibrant objects in $\caa$.

\item
Suppose that $(\cyl,\gamma,\sigma)$ is cartesian and that the
right adjoint of $\cyl$ leaves $\caa$ invariant.
Then the cylinder $(R\cyl,R\gamma,R\sigma)$ is also cartesian.

\item
Given $S\subseteq \lbox{(\rbox I)}$, if in the situation of (c)
every object of $\caa$ is fibrant w.r.t.~$\Lambda(\cyl,S,I)$
then $\cww(R\cyl,RS,RI) = \caa \intersection \cww(\cyl,S,I)$.
\end{alphlist}
\end{lemma}

\begin{proof}
First observe that by Part~(a) of Theorem~\ref{thm:lp_is_wfs} the set
$RI$ indeed generates a weak factorization system in $\caa$, which is
cofibrant because $\caa$ is full.
We will repeatedly use the equivalence
$$
Rf \boxrel g \iff f\boxrel g  \text{\qquad for all } f\in\ckk, g\in\caa
\eqno{(*)}
$$
which holds by adjointness between $R$ and the inclusion of $\caa$.
Given any object $A\in\caa$, its coproduct with itself in $\caa$ is
$R(A+A)$ and also $RA\iso A$.
Application of $R$ to diagram~\eqref{dia:cyl} in
Definition~\ref{def:htop} therefore shows that $RCA$
is indeed a cylinder object for $A$.
\begin{alphlist}
\item
Consider any two maps $\map{f,g}AB$ in $\caa$ and the induced map
$\map {(f|g)}{A+A}B$ from the coproduct in $\ckk$.
\newline
Then $\map{\widehat{(f|g)}}{R(A+A)}B$ is the induced map from
the coproduct in $\caa$.
Now $f \htop g \pmod{\cyl} \iff f \htop g \pmod{R\cyl}$ follows
with ($*$).

\item
Because $R$ preserves pushouts, we have
$
Rf\star R\gamma = R(f\star\gamma)
$
and
$
Rf\star R\gamma^k = R(f\star\gamma^k)
$
(for $k=0,1$),
which gives the equality $\Lambda(R\cyl,RS,RI) = R\Lambda(\cyl,S,I)$.
By ($*$) we have
\newline
$
R\Lambda(\cyl,S,I) \boxrel (A\to1)
\iff 
\Lambda(\cyl,S,I) \boxrel (A\to1)
$
and hence $\Lambda(\cyl,S,I)$ and $\Lambda(R\cyl,RS,RI)$ determine the
same class of fibrant objects.

\item
Let $\map G\ckk\ckk$ be a right adjoint of $\cyl$ with $G\caa\subseteq\caa$.
The isomorphisms (natural in $A,B\in\caa$)
$$\caa(R\cyl A,B) \iso \ckk(\cyl A,B) \iso \ckk(A,GB) \iso \caa(A,GB)$$
exhibit the cylinder functor as a left adjoint.
The second condition in Definition~\ref{def:cart} holds because of (b).

\item
By Corollary~\ref{cor:heq} and part (a) above, both $\cww(R\cyl,RS,RI)$ and
$\caa \intersection \cww(\cyl,S,I)$ coincide with the class of homotopy
equivalences in $\caa$.\qedhere
\end{alphlist}
\end{proof}

In the situation of the above Lemma, one cannot expect in general
that a final cylinder on $\ckk$ will induce a final cylinder on
the subcategory $\caa$. Therefore the induced model structure
may fail to be left determined even if the original one was.
Nevertheless, in the next three examples one can check directly
that the induced cylinders are final and hence the induced model
structures are left determined.

\begin{example}
\label{ex:prord}
Let $\ckk=\Cat$ and $\caa=\PrOrd$, the category of preordered sets
(i.e.~sets with a reflexive and transitive relation) and monotone maps.
$\PrOrd$ has a model structure where the cofibrations are the
monomorphisms and the weak equivalences are the categorical equivalences.
We will obtain it from the previous one on $\Cat$.

The reflection $\map R\Cat\PrOrd$ is bijective on objects and identifies
parallel arrows. We will keep the notation from Example~\ref{ex:cat}.
Discarding the isomorphism $Rp$ from $RI$, we obtain the generating set
$I'= RI\setminus\{Rp\} = \{(0\to1), (2\inclusion\underline2)\}$.
One has $\lbox{(\rbox{I'})}=\Mono$, which is obtained exactly as
in Example~\ref{ex:cat}, keeping in mind that functors between preorders
are always faithful and that the monomorphisms in $\PrOrd$ are exactly
the functors that are injective on objects.
The right adjoint to $(-)\times\overline{\underline2}$ is
$(-)^{\overline{\underline2}}$ which leaves $\PrOrd$ invariant.
Every object is fibrant and therefore $\cww'=\cww(\emptyset,I')$ consists
of the categorical equivalences.
\end{example}

\begin{example}
Let $\ckk=\PrOrd$ and $\caa=\Ord$, the category of ordered sets
(i.e.~sets with a reflexive, transitive and antisymmetric relation)
and monotone maps.
$\Ord$ has a model structure where the cofibrations are all maps
and the weak equivalences are the isomorphisms.
We will obtain it from the previous one on $\PrOrd$.

The reflection $\map R\PrOrd\Ord$ assigns to every preordered set $X$
the quotient $X/{\htop}$ obtained from identifying homotopic elements.
The generating set $I'=\{(0\to1), (2\inclusion\underline2)\}$ is already
contained in $\Ord$ and hence $I'=RI'$.
Because a full surjective functor between ordered sets must be an isomorphism,
the class $\rbox I$ consists of all isomorphisms and consequently
$\lbox{(\rbox I)}=\Ord$. For any ordered set $P$ one has
$P^{\overline{\underline2}}=P$. Therefore $\Ord$ is invariant under
$(-)^{\overline{\underline2}}$.
Every object is fibrant and therefore $\cww'=\cww(\emptyset,I')$ is
the class of isomorphisms.
\end{example}

\begin{example}
Let $\ckk=\PrOrd$ and $\caa=\Set$. Here we identify $\Set$ with
the full subcategory of indiscrete preordered sets.
It has a model structure where the cofibrations are the
monomorphisms and the weak equivalences are the maps between nonempty
sets together with the identity map of the empty set.
This (almost trivial) model structure is also mentioned in
\cite{CISI02}*{Exemple~3.7} and \cite{ROTH03}*{Section~3}.
It can be constructed with the cylinder in Example~\ref{ex:topos},
with the set of generating cofibrations given by
the proof in \cite{BEKE00}*{Proposition~1.12}.
Instead we will obtain it from the one on $\PrOrd$
in Example~\ref{ex:prord}.

The reflection $\map R\PrOrd\Set$ assigns to every preordered set
the indiscrete preorder on its elements.
Let $I'$ be as in Example~\ref{ex:prord}.
Discarding the identity map $\overline{\underline2}$ from $RI'$,
we obtain the generating set $I''= \{(0\to1)\}$ in $\Set$.
Then $\rbox{I''}$ is the class of surjective maps and
$\lbox{(\rbox{I''})}=\Mono$. 
For any indiscrete preorder $X$, the preorder $X^{\overline{\underline2}}$
is again indiscrete. Therefore $\Set$ is invariant under
$(-)^{\overline{\underline2}}$.
Every object is fibrant and therefore $\cww''=\cww(\emptyset,I'')$
consists of the identity map of the empty set and of all maps with nonempty
domain.
\end{example}

\noindent
In the previous examples, all objects were fibrant and consequently
the homotopy relation already determined the weak equivalences via
Corollary~\ref{cor:heq}. Here is an example where this does not happen.

\begin{example}
Let $\ckk=\rsRel$, the category of plain undirected graphs
(i.e. sets with a reflexive and symmetric relation together with
maps preserving such relations).
We will construct a left determined model structure on $\rsRel$ where
the cofibrations are the monomorphisms and the weak equivalences
are those maps that induce bijections between path components.
It can be seen as the one-dimensional version of the
left determined model structure on simplicial complexes
as described in \cite{ROTH03}*{Remark~3.7}.

We will write $n$ for the discrete graph on $n$ vertices,
$K_n$ for the indiscrete (i.e. complete) graph on $n$ vertices
and $K_n^-$ for the graph obtained from $K_n$ by deleting one edge.
\newline
Consider the set $I=\{(0\to1), (2\inclusion K_2)\}$, where
the second map is the usual inclusion.
\begin{numlist}

\item
We first check that $I$ is a set of generating cofibrations.
The class $\rbox I$ consists of those maps
$\map f{(X,\alpha)}{(Y,\beta)}$ that are
surjective and full (i.e. satisfy $f(x)\beta f(x')\follows x\alpha x'$).

Moreover one has $\lbox{(\rbox I)}=\Mono$.
This follows by the same argument as in the case of categories 
(step~(1) in Example~\ref{ex:cat}) with $K_2$ in place of
$\overline{\underline2}$.

\item
The cylinder functor $\cyl=(-)\times K_2$
is obtained from the factorization $2\inclusion K_2\to1$
and $\map{\gamma_X}{X\times 2}{X\times K_2}$
is the usual inclusion. Because $(K_2\to1)$ is in
$\rbox I$, the resulting cylinder is final.
Two vertices $\map{x,y}1X$ of a graph
are homotopic iff they are joined by an edge in $X$.
Therefore, for two maps $\map{f,g}{(X,\alpha)}{(Y,\beta)}$ one has
$$
f\htop g
\iff \forall x,x'\in X: \left( x\alpha x' \follows f(x)\beta g(x')\right)
$$
because $Y^X$ is $\rsRel(X,Y)$ equipped with the relation $\beta^\alpha$
defined by the condition on the right side of the above equivalence.
In particular the homotopy relation is not transitive in general.
The homotopy relation on $\rsRel(X,Y)$ is transitive whenever $Y$
(i.e. its relation) is transitive.
Moreover, if $Y$ is discrete then homotopy coincides with equality.

\item
For a partial description of $\Lambda=\Lambda(\emptyset,I)$
first observe, that the forgetful functor $\rsRel\to\Set$
preserves pushouts. In particular, in a pushout diagram
$$
\xymatrix{
A\times 2 \ar[r]^{\gamma_A} \ar[d]_{f\times2}
& A\times K_2 \ar[d] \\
B\times 2 \ar[r] & Q
}
$$
one can assume that the underlying set of $Q$ is $B\times2$,
that the horizontal underlying maps are identity maps and that
the two vertical underlying maps coincide.
Now suppose that $A$ is nonempty and $B$ is indiscrete.

\noindent
Then $Q$ is path connected: given any $b,b'\in B$ and $i,j\in2$,
take some $a\in A$ with $\xymatrix{b\edge[r]& f(a)\edge[r]& b'}$.
Then
\begin{romlist}
\item
\quad$\xymatrix@1{(b,i) \edge[r] & (f(a),i)}$ in $B\times2$

\item
\quad$\xymatrix@1{(a,i) \edge[r] & (a,j)}$ in $A\times K_2$

\item
\quad$\xymatrix@1{(f(a),j) \edge[r] & (b',j)}$ in $B\times2$
\end{romlist}
and passing to $Q$ gives a path
$$
\xymatrix{(b,i) \edge[r]& (f(a),i) \edge[r]& (f(a),j) \edge[r]& (b',j)}
$$
in $Q$. Hence, if $\map fAB$ is an inclusion then 
$f\star\gamma$ is the inclusion of the (nonempty)
path connected $Q$ into the indiscrete $B\times K_2$.

\noindent
As in Example~\ref{ex:cat} we have
$(0\to1)\star\gamma^k=\map{\gamma^k_1}1{K_2}$.
From the inclusion $\map{\gamma_1}2{K_2}$ we obtain the following diagram
$$
\xymatrix{
2 \ar[r]^-{\gamma^0_2} \ar[d]_{\gamma_1}
 & 2\times K_2 \ar[d] \ar@/^/[ddr]^{\gamma_1\times K_2} & \\
K_2 \ar[r] \ar@/_/[drr]_{\gamma^0_{K_2}} &
K_4^- \ar[dr]|-{\gamma_1\star\gamma^0} & \\
 & & K_2\times K_2
}
$$
where (according to the notation introduced) $K_4^-$ is the graph
$$
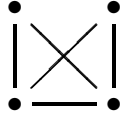
\xymatrix{
{\bullet} \edge[d] \edge[dr] & {\bullet} \edge[d] \\
{\bullet} \edge[r] \edge[ur] & {\bullet} 
}
$$
and $\gamma_1\star\gamma^0$ is the inclusion of $K_4^-$
into $K_4=K_2\times K_2$.
Up to a permutation of vertices, the same inclusion is obtained
as $\gamma_1\star\gamma^1$.

Hence each map in $\Lambda^0$ is the inclusion of a nonempty
path connected subgraph of some suitable $K_n$
Applying the above observation gives (via induction) that each $\Lambda^n$
consists only of maps of this type. Except for the two inclusions
$\gamma^0_1$ and $\gamma^1_1$, the included subgraph of $K_n$ is wide,
i.e. it has the maximal number of vertices.

Consequently, every transitive graph $T$ is fibrant:
given some inclusion $P\inclusion K_n$ with $P$ path connected
and $|P|=n$, any map $\map fPT$ can be extended to $\map h{K_n}T$
by $h(x):=f(x)$.

Conversely, assume that $X$ is fibrant.
Observe that ${K_3^-\inclusion K_3}$ is in $\lbox{(\rbox\Lambda)}$
because it can be obtained from ${K_4^-\inclusion K_4}$ as a pushout
$$
\xymatrix{
K_4^- \ar@{_{(}->}[d] \ar[r]^p & K_3^- \ar@{_{(}->}[d]\\ 
K_4 \ar[r] & K_3
}
$$
where $p$ is the surjection that collapses the two vertices of degree $3$.
Therefore, every map $\map f{K_3^-}X$ can be extended to a map
$\map {f'}{K_3}X$, which is precisely the definition of transitivity.

In summary, the fibrant graphs are exactly the transitive graphs.

\item
For a graph $(X,\alpha)$, a path component is an equivalence
class of the transitive closure $\alpha^*$ of the relation $\alpha$.
We write $[x]$ for the equivalence class of any $x\in X$ and
$\pi_0X$ for the discrete graph on the set $\{[x] \mid x \in X\}$.
Setting $\pi_0f([x]):=[f(x)]$ for any $\map fXY$ makes
$\pi_0$ into a functor and the canonical map
$\map {r_X}X{\pi_0X}$ with $r(x)=[x]$ gives a reflection into
the subcategory of discrete graphs. 
For two maps $\map{f,g}{(X,\alpha)}{(Y,\beta)}$ one has:
$$
\pi_0f=\pi_0g
\iff \forall x,x'\in X: \left( x\alpha^* x' \follows f(x)\beta^* g(x') \right)
$$
Comparing this with the homotopy condition
$$
f\htop g
\iff \forall x,x'\in X: \left( x\alpha x' \follows f(x)\beta g(x') \right)
$$
one obtains that always $f\htop g \follows \pi_0f=\pi_0g$ and that
the converse implication $\pi_0f=\pi_0g \follows f\htop g$ holds
whenever $\beta$ is already transitive.
In the general case of a map $\map fXY$ one has:
$$
f \in \cww \iff \pi_0f \text{ is an isomorphism }
$$
For the direction ''$\Rightarrow$'' assume $f\in \cww$.
Remark~\ref{rem:heq} with $t=r_X$ and $T=\pi_0X$
gives a map $\map uX{\pi_0X}$ such that in the diagram
$$
\xymatrix{
X \ar[d]_{r_X} \ar[r]^f & Y \ar[d]^{r_Y} \ar[ld]^u  \\
\pi_0X \ar[r]_{\pi_0f} & \pi_0Y
}
$$
we have $r_X\htop fu$. Then also $fr_Y=r_X(\pi_0f)\htop fu(\pi_0f)$
and by Remark~\ref{rem:heq} with $t=r_X(\pi_0f)$ and $T=\pi_0Y$
this forces $r_Y\htop u(\pi_0f)$. But for discrete codomains,
homotopy means equality and hence the above diagram strictly commutes.
Applying the functor $\pi_0$ to that diagram exhibits $\pi_0u$ as the
two-sided inverse of $\pi_0f$.

For the direction ''$\Leftarrow$'' assume that $\pi_0f$ is an isomorphism
and let $\map tXT$ be a map to a transitive graph $T$.
Uniqueness up to homotopy follows from the equivalence
\begin{align*}
fh\htop fh'
&\iff 
(\pi_0f)(\pi_0h) = (\pi_0f)(\pi_0h') \\
&\iff 
\pi_0h = \pi_0h'
\iff 
h\htop h'
\end{align*}
for any $\map {h,h'}YT$ because $T$ is transitive.

For existence, let $\map s{\pi_0T}T$ be a section of $r_T$
with $\pi_0s=\pi_0T$ (i.e. a choice of representatives of the path 
components) and define $\map hYT$ as the composite
$h= r_Y (\pi_0f)^{-1}(\pi_0t)s$.
Then $\pi_0(fh)=\pi_0t$ and hence $fh\htop t$.
\end{numlist}
\end{example}

\begin{example}
\label{ex:eqr}
Keep the notation of the previous example and consider
the full reflective subcategory $\eqRel$ of transitive graphs,
i.e.~sets equipped with an equivalence relation.
It has a model structure where the cofibrations are
the monomorphisms and the weak equivalences are those maps that
induce bijections between equivalence classes. This model structure has
been described in detail by L\'arusson~\cite{LARU06}.
We will obtain it via Lemma~\ref{lem:induced} from the previous
one on $\rsRel$.

The reflection $\map R{\rsRel}{\eqRel}$ assigns to every graph $(X,\alpha)$
its transitive closure $(X,\alpha^*)$. Because the graphs $0$, $1$, $2$
and $K_2$ are already transitive, one obtains $RI=I$ and
also $\lbox{(\rbox{RI})}=\Mono\intersection\eqRel$ as in step (1) above.
Moreover, if $X$ is transitive then so is $X^{K_2}$ and we already
noted in step (3) that all transitive graphs are fibrant.
From Lemma~\ref{lem:induced} we now obtain that $\cww'=\cww(\emptyset,I)$
consists of those maps $f$ where $\pi_0f$ is an isomorphism, i.e.~those
maps that induce a bijection between equivalence classes.
Finally observe, that $R$ preserves full surjections.
Therefore the induced cylinder is again final and the induced model
structure is left determined.
\end{example}

We now turn from ''space-like'' to ''linear'' examples.
Let $R$ be a ring and let $\ckk={}_R\Mod$, the category of left
$R$-modules. We also write $\Mod_R$ and ${}_R\Mod_R$ for the categories
of right and two-sided $R$-modules respectively.
We always have a cofibrant weak factorization system $(\Mono,\rbox\Mono)$
in $\ckk$, which is generated by the set $I$ of all inclusions
$\ideal{a}\inclusion R$ of left ideals. Also $\rbox\Mono$ consists
of all those epimorphisms with injective kernel
(for details see \cite{AHRT02b}*{Example~1.8(i)}).

We will only be concerned with model structures constructed from
the above weak factorization system, i.e. where $\Mono$ is the
class of cofibrations. Hence it remains to find cartesian cylinders.

In order to find possible examples, we first characterize cartesian
cylinders for the weak factorization system $(\Mono,\rbox\Mono)$ in $\ckk$.
Recall that a map $\map fUV$ of right modules is $\df{pure}$
(or equivalently that $f(U)$ is a pure submodule of $V$) if for every
(finitely generated) left module $M$, the map
$\map{f\tensor_R M}{U\tensor_R M}{V\tensor_R M}$ is a monomorphism.
We use another characterization of pure submodules:
$U\subseteq V$ is pure iff every finite system of equations
$$
u_j=\sum_i x_i r_{ij} \quad(u_j\in U,\, r_{ij}\in R)
$$
which has a solution with $x_i\in V$ also has a solution with $x_i\in U$.
For a direct proof, which can easily be adapted to the non-commutative
setting, see e.g.~Matsumura~\cite{MATS89}*{Theorem~7.13}.

\begin{proposition}
\label{prop:modules}
Suppose $V$ is a two-sided $R$-module together with a map $\map vRV$
in ${}_R\Mod_R$ and let $\map{\cyl_v}\ckk\ckk$ be the functor with
$\cyl_v(M)=(R+V)\tensor_R M = M+V\tensor_RM$.
Let $\map{\gamma^0_R}R{R+V}$ be the coproduct injection,
$\map{\sigma_R}{R+V}R$ be the product projection
and $\map{\gamma^1_R=(R,v)}R{R+V}$.
Set $\sigma=\sigma_R\tensor_R(-)$ and
$\gamma=(\gamma^0_R|\gamma^1_R)\tensor_R(-)$.
Then the following holds:

\begin{alphlist}

\item
$(\cyl_v,\gamma,\sigma) \text{ is a cylinder}
\iff
\map vRV \text{ is pure (in $\Mod_R$)}$.
Moreover, two maps $\map {f,g}MN$ are homotopic iff
\newline
$\map{g-f}MN$ factors through $\map {v\tensor_RM}M{V\tensor_RM}$.

\item
Every left adjoint cylinder $(\cyl,\gamma,\sigma)$ arises in this
way from a suitable pure monomorphism $\map vR{\ker(\sigma_R)}$.

\item
In the situation of (a) we have
\newline
$(\cyl_v,\gamma,\sigma) \text{ is cartesian}
\iff
V \text{ is a flat right module}$.

\item
In the situation of (a) we have
\newline
$(\cyl_v,\gamma,\sigma) \text{ is final}
\iff
V\tensor_RM \text{ is injective for every } M$.
\end{alphlist}
\end{proposition}

\begin{proof}
We use familiar matrix notation for maps between (co)products
and omit the object names for identities and zero maps.
Then the maps introduced above can be written as
$\gamma^0_R={(\begin{smallmatrix} 1&0\end{smallmatrix})}$,
$\gamma^1_R={(\begin{smallmatrix} 1&v\end{smallmatrix})}$,
$\gamma_R={\bigl(\begin{smallmatrix} 1&0\\1&v\end{smallmatrix}\bigr)}$
and
$\sigma_R={\bigl(\begin{smallmatrix}1\\0\end{smallmatrix}\bigr)}$.
Abbreviating  $v\tensor_RM$ as $v_M$ and $V \tensor f$ as $f_V$,
we can also write
$\gamma_M={\bigl(\begin{smallmatrix} 1&0\\1&v_M\end{smallmatrix}\bigr)}$
and
$\cyl(f) = {\bigl(\begin{smallmatrix} f&0\\0&f_V\end{smallmatrix}\bigr)}$.

\begin{alphlist}

\item
Because of
$
\bigl(\begin{smallmatrix} 1&0\\1&v_M\end{smallmatrix}\bigr)
\bigl(\begin{smallmatrix}1\\0\end{smallmatrix}\bigr)
=
\bigl(\begin{smallmatrix}1\\1\end{smallmatrix}\bigr)
$
the maps  $\gamma_M$ and $\sigma_M$ clearly factor the codiagonal.
Moreover, $\gamma_M$ is a monomorphism iff $v_M$ is a monomorphism,
from which the equivalence follows.

Given two maps $\map{f,g}MN$, the map
$\map{\bigl(\begin{smallmatrix}f\\ g\end{smallmatrix}\bigr)}{M+M}N$
can be extended along $\map{\gamma_M}{M+M}{M+V\tensor_R M}$ iff the
\newline
equation
$$
\begin{pmatrix} 1 & 0\\ 1 & v_M \end{pmatrix}
\begin{pmatrix} h_1\\ h_2 \end{pmatrix}
=\begin{pmatrix} f\\ g \end{pmatrix}
$$
can be solved with some $\map{h_1}MN$ and $\map{h_2}{V\tensor_R M}N$.
This is equivalent to the condition that $\map{g-f}MN$
extends along $\map {v_M}M{V\tensor_R M}$.

\item
Let $(\cyl,\gamma,\sigma)$ be a cylinder such that $\cyl$ has a right
adjoint $G$.

Application of $\cyl$ to the right translations $\map {\rho_r}RR$ for
each $r\in R$ gives a right action of $R$ on $\cyl(R)$ which makes
$\cyl(R)$ into a two-sided module such that the isomorphisms
$$
\ckk(\cyl(R),M) \iso \ckk(R,G(M)) \iso G(M)
$$
are isomorphisms of left modules and hence
$\cyl \iso \cyl(R)\tensor_R(-)$.
Moreover, the diagrams
$$
\xymatrix{
R \ar[r]^-{\gamma^k_R} \ar[d]_{\rho_r}
& {\cyl(R)} \ar[r]^-{\sigma_R} \ar[d]^{\cyl(\rho_r)} & R \ar[d]^{\rho_r} \\
R \ar[r]^-{\gamma^k_R} & {\cyl(R)} \ar[r]^-{\sigma_R} & R 
}
$$
show that $\sigma_R$ and the $\gamma^k_R$ are maps of two-sided modules.
Consequently, $\cyl(R) = \gamma^0_R(R) + \ker(\sigma_R)$
is a decomposition as two-sided modules.
With respect to this decomposition, we obtain

\noindent
$\gamma^0_R={(\begin{smallmatrix} 1&0\end{smallmatrix})}$,
and
$\sigma_R={\bigl(\begin{smallmatrix}1\\0\end{smallmatrix}\bigr)}$.
Moreover,
$\gamma^1_R={(\begin{smallmatrix} 1&v\end{smallmatrix})}$
for some $\map vR{\ker(\sigma_R)}$.
Application of naturality of $\gamma$ and $\sigma$ to an $\map mRM$
then gives
$\gamma_M=\gamma_R\tensor_R M$ and $\sigma_M=\sigma_R\tensor_R M$.

\item
Let $\map iMN$ be a monomorphism.

The pushout of $i$ and $\gamma^0_M$ is
$N+V\tensor_R M$ and $i\star\gamma^0$ is the map
$\map%
{\bigl(\begin{smallmatrix} 1&0\\0&i_V\end{smallmatrix}\bigr)}
{N+V\tensor_R M}
{N+V\tensor_R N}$.
Therefore $i\star\gamma^0$ is a monomorphism iff $i_V$ is a monomorphism.
In particular, flatness of $V$ is necessary for $(\cyl_v,\gamma,\sigma)$
to be cartesian.

Now suppose $V$ is flat.
As seen above, $i\star\gamma^0$ is a monomorphism.
Because of
$
(\begin{smallmatrix} 1&v_M \end{smallmatrix})
=
(\begin{smallmatrix} 1&0 \end{smallmatrix})
\bigl(\begin{smallmatrix} 1&v_M \\ 0&1 \end{smallmatrix}\bigr)
$
the maps $\gamma^0_M$ and $\gamma^1_M$ differ only by an automorphism
of their codomain. Moreover, for any $\map fMN$ one has
$v_Mf_V = v\tensor_Rf = fv_N$ and hence
$
\bigl(\begin{smallmatrix} 1&v_M \\ 0&1 \end{smallmatrix}\bigr)
\bigl(\begin{smallmatrix} f&0 \\ 0&f_V \end{smallmatrix}\bigr)
=
\bigl(\begin{smallmatrix} f&0 \\ 0&f_V \end{smallmatrix}\bigr)
\bigl(\begin{smallmatrix} 1&v_N \\ 0&1 \end{smallmatrix}\bigr)
$.
Therefore these automorphisms are part of a natural automorphism
on the cylinder functor.
Consequently $i\star\gamma^1$ is the pushout of $i\star\gamma^0$
along an isomorphism and hence $i\star\gamma^1$ is also a monomorphism.

For $i\star\gamma$, it is enough to consider the special
case where $i$ is the inclusion $\ideal{a}\inclusion R$ of a
left ideal. Let $\map j{V\tensor_R \ideal{a}}V$ be the map
with $j(w\tensor a)=wa$.
The pushout $Q$ of $i$ and $\gamma_M$ can be calculated as the
cokernel in the exact row below
$$
\xymatrix{
0 \ar[r] & \ideal{a}+\ideal{a}
\ar[r]^-k
& R+R+\ideal{a}+V\tensor_R \ideal{a} \ar[r]
\ar[d]_h
& Q \ar[r] \ar[ld]^{i\star\gamma} &0 \\
&& R+V &&
}
$$
where
$$
k= \begin{pmatrix}
-i & 0 & 1 & 0\\ 0 & -i & 1 & v_{\ideal{a}}
\end{pmatrix}
\text{ \quad and \quad }
h= \begin{pmatrix}
1 & 1 & i & 0\\ 0 & v & 0 & j
\end{pmatrix}^{\top}
$$
and $i\star\gamma$ is induced by $h$ because $\image(k)\subseteq\ker(h)$.
To show that $i\star\gamma$ is a monomorphism, it remains to
verify $\ker(h)\subseteq\image(k)$.

Assume $(x,y,a,w)\in\ker(h)$ for some $x,y\in R$, $a\in\ideal{a}$ and
$w= \sum_n w_n\tensor b_n \in V\tensor_R \ideal{a}$.
This corresponds to equations $x+y+a=0$ and $-vy = \sum_n w_nb_n$.
Because $vR$ is a pure submodule of $V$, there are $r_n\in R$ with
$-vy = \sum_n vr_nb_n$. Since $v$ is a monomorphism, we have
$y = -\sum_n r_nb_n \in\ideal{a}$ and $x\in\ideal{a}$.
Therefore
$(x,y,a,w)= (-x,-y)\bigl(\begin{smallmatrix}
-i & 0 & 1 & 0\\ 0 & -i & 1 & v_{\ideal{a}}
\end{smallmatrix}\bigr)\in \image(k)$.

\item
Tensoring the split exact sequence
$$
\xymatrix{
0 \ar[r]
& V \ar[r]^-{(\begin{smallmatrix} 0 & 1 \end{smallmatrix})}
& R+V \ar[r]^-{\bigl(\begin{smallmatrix} 1 \\ 0 \end{smallmatrix}\bigr)}
& R \ar[r] & 0
}
$$
with $M$, we obtain $\ker(\sigma_M)=V\tensor_R M$ from which the equivalence
follows.\qedhere
\end{alphlist}
\end{proof}

\noindent
Observe that in the situation of \ref{prop:modules}(d),
two maps $\map {f,g}MN$ are homotopic iff $\map{g-f}MN$ factors
through some injective module. This relation is known as stable
equivalence (see e.g.~\cite{KURO05}*{Section~4}
or \cite{HOVE99}*{Definition~2.2.2}) and the homotopy equivalences
will then also be called stable equivalences.

\begin{corollary}
\label{cor:modules}
Let $(\cyl,\gamma,\sigma)$ be a final cartesian cylinder in ${}_RMod$
and suppose that the ring $R$ is injective.
Then each map in $\Lambda=\Lambda(\cyl,\emptyset,I)$ has injective domain
and codomain. In particular, every object is fibrant and
$\cww=\cww(\cyl,\emptyset,I)$ is the class of stable equivalences.
\end{corollary}

\begin{proof}
By part (b) of Proposition~\ref{prop:modules} one can assume $\cyl=\cyl_v$
for some $\map vRV$.
Moreover, $\cyl_v$ preserves injective objects by part (d).
We prove by induction that each map in $\Lambda^n$ has injective
domain and codomain.

For an inclusion $\map i{\ideal{a}}R$ of a left ideal,
we already remarked in the proof of part (c) that $i\star\gamma^0$
and $i\star\gamma^1$ have isomorphic domains and also calculated
$\map {i\star\gamma^0}{R+V\tensor_R\ideal{a}}{R+V\tensor_R R}$.
Therefore every map in $\Lambda^0$ has injective domain and codomain.

Now assume that the claim holds for $\Lambda^n$
and let $\map fMN$ be a map in $\Lambda^n$. Then the codomain of
$f\star\gamma$ is $N+V\tensor_R N$, which is injective.
Its domain $Q$ is the cokernel of a split exact sequence
$$
\xymatrix{
0 \ar[r] & M+M \ar[r] & N+N+M+V\tensor_R M \ar[r] & Q \ar[r] & 0
}
$$
and is therefore also injective.
\end{proof}

\begin{example}
Let $H$ be a finite dimensional Hopf algebra over a field $k$, i.e.~a
(finite dimensional) $k$-algebra together with algebra maps
$\map\Delta H{H\tensor_k H}$ (comultiplication) and $\map\varepsilon Hk$
(counit), and an anti-algebra map $\map SHH$ (antipode) satisfying
certain conditions (for details see e.g.\ Montgomery~\cite{MONT93}).
${}_H\Mod$ has a model structure where the weak equivalences are
the stable equivalences
\cite{HOVE99}*{Theorem~2.2.12 and Proposition~4.2.15}
We will show that this model structure is left determined by
verifying the conditions of Proposition~\ref{prop:modules} and
Corollary~\ref{cor:modules}.

\begin{numlist}
\item
Due to results of Larson and Sweedler~\cite{LASW69}*{Theorem~2 (p79)
and Proposition~2 (p83)} on finite dimensional Hopf algebras over a field,
$H$ satisfies the following conditions:
\begin{alphlist}

\item the antipode $\map SHH$ is invertible.

\item there exists a nonzero $d\in H$ with $hd=\varepsilon(h)d$
for all $h\in H$.
Giving $k$ a left $H$-module structure via $\map\varepsilon Hk$, such a $d$
corresponds to a (nonzero) $H$-linear map $\map dkH$.

\item a left $H$-module is injective iff it is projective
\end{alphlist}
\item
Let $M$ and $N$ be two $H$-modules.
Then $M\tensor_k N$ has an $H\tensor_k H$-module structure with
$(c\tensor c')(m\tensor n)=cm\tensor c'n$.
Via the map $\map\Delta H{H\tensor_k H}$ this induces an $H$-module structure
on ${M\tensor_k N}$. Observe that with this definition
$k\tensor_k M \iso M \iso M\tensor_k k$ and for a two sided module $V$
also $M\tensor_k (V \tensor_H N) \iso (M\tensor_k V) \tensor_H N$
as $H$-modules.

Let $\Hom(M,N)$ be the group of all $k$-linear maps from $M$ to $N$.
Then $\Hom(M,N)$ has a $H\tensor_k H^{op}$-module structure with
$((c\tensor c')f)m=c(f(c'm))$. From this one obtains two different
$H$-module structures on $\Hom(M,N)$:

The first one is induced via
$\xymatrix@1{
H \ar[r]^-\Delta & H\tensor_k H \ar[r]^-{H\tensor S} & H\tensor_k H^{op}}$.
We write $\Hom^r(M,N)$ for this module structure.
The second one is induced via
$\xymatrix@1{
H \ar[r]^-\Delta & H\tensor_k H \ar[r]^-{tw}
& H\tensor_k H \ar[r]^-{\,H\tensor S^{-1}} & H\tensor_k H^{op}}$,
where $tw$ is defined by $tw(c\tensor c')=c'\tensor c$.
We write $\Hom^l(M,N)$ for this module structure.

Then one can verify that this gives bifunctors on ${}_H\Mod$ and
that for any given $M$, the $k$-linear evaluation maps
\newline
$\map {e_N}{\Hom^r(M,N)\tensor_k M}N$ and
$\map {e'_N}{M\tensor_k\Hom^l(M,N)}N$
defined by $e_N(f,m) = fm = e'_N(m,f)$
are indeed $H$-linear and provide counits of two adjunctions
$(-)\tensor_k M\adjoint\Hom^r(M,-)$ and
$M\tensor_k (-)\adjoint\Hom^l(M,-)$.

\item
We fix some $\map dkH$ as in (1b) above.
Set $V= H\tensor_k H$. Then $V$ is a two sided $H$-module.
Define $\map vHV$ by the composition
$\xymatrix@1{H \iso k\tensor_k H \ar[r]^-{d\tensor H} & H\tensor_k H}$.
Then this gives a map of two sided $H$-modules.

\item
Tensoring over the field $k$ with a fixed module preserves monomorphisms.
In particular the above $\map vHV$ is a monomorphism.
Moreover the natural isomorphisms $v\tensor_H(-) \iso d\tensor_k(-)$
and $V\tensor_H(-) \iso H\tensor_k(-)$ yield that
$\map vHV$ is pure and $V$ is flat.

\item
For a fixed module $M$, both $\Hom^l(M,-)$ and $\Hom^r(M,-)$
preserve epimorphisms. Therefore their left adjoints $M\tensor_k(-)$ and
$(-)\tensor_k M$ preserve projective $H$-modules.
In particular, $V\tensor_H M \iso H\tensor_k M$ is projective and
therefore injective.
\end{numlist}
\end{example}

\end{document}